\documentclass[reqno,a4paper,11pt]{article}

\usepackage{epic,eepic,epsfig,amssymb,amsmath,amsfonts,amsthm,mathrsfs}

\title{Mass Transportation on Sub-Riemannian Manifolds}

\author{A.~Figalli\footnote{Universit\'e de Nice-Sophia
    Antipolis, Labo.\ J.-A.\ Dieudonn\'e, UMR 6621, Parc
    Valrose, 06108 Nice Cedex 02, France ({\tt
      figalli@unice.fr})}
      \footnote{Centre de Math\'ematiques Laurent Schwartz, Ecole Polytechnique, 91128 Palaiseau, France ({\tt figalli@math.polytechnique.fr})}
\and
L.~Rifford\footnote{Universit\'e de Nice-Sophia
    Antipolis, Labo.\ J.-A.\ Dieudonn\'e, UMR 6621, Parc
    Valrose, 06108 Nice Cedex 02, France ({\tt
      rifford@unice.fr})}}

\date{}
\numberwithin{equation}{section}

\newtheorem{theorem}{Theorem}[section]

\newtheorem{proposition}[theorem]{Proposition}
\newtheorem{lemma}[theorem]{Lemma}
\theoremstyle{definition}\newtheorem{remark}[theorem]{Remark}
\theoremstyle{definition}\newtheorem{definition}[theorem]{Definition}

\theoremstyle{definition}

\def\R{\textrm{I\kern-0.21emR}}
\def\N{\textrm{I\kern-0.21emN}}

\def\Bad{\mathcal{B}}
\def\Mov{\mathcal{M}}
\def\Stat{\mathcal{S}}
\def\p{\partial}
\def\g{\gamma}

\def\O{\Omega}
\def\e{\varepsilon}

\newcommand{\SPAN}{{\rm Span}}
\newcommand{\supp}{{\rm \operatorname{supp}}}
\newcommand{\Hess}{{\rm \operatorname{Hess\,}}}
\renewcommand{\exp}{{\rm exp}}

\begin{document}

\maketitle

\begin{abstract}
We study the optimal transport problem in sub-Riemannian manifolds
where the cost function is given by the square of the
sub-Riemannian distance. Under appropriate assumptions, we
generalize Brenier-McCann's Theorem proving existence and
uniqueness of the optimal transport map. We show the absolute
continuity property of Wassertein geodesics, and we address the
regularity issue of the optimal map. In particular, we are able to
show its approximate differentiability a.e. in the Heisenberg
group (and under some weak assumptions on the measures the
differentiability a.e.), which allows to write a weak form of the
Monge-Amp\`ere equation.
\end{abstract}

\section{Introduction}

The optimal transport problem can be stated as follows: given two
probability measures $\mu$ and $\nu$, defined on measurable
spaces $X$ and $Y$ respectively, find a measurable map $T:X \rightarrow Y$ with
$$
T_\sharp \mu=\nu \qquad \text{(i.e. $\nu(A)=\mu\left(T^{-1}(A)\right)$ for all $A \subset Y$ measurable)},
$$
and in such a way that $T$ minimizes the transportation cost. This
last condition means
$$
\int_X c(x,T(x))\,d\mu(x) = \min_{S_\sharp \mu=\nu} \left\{ \int_X
c(x,S(x))\,d\mu(x) \right\},
$$
where $c:X \times Y \rightarrow \R$ is some given cost function,
and the minimum is taken over all measurable maps $S:X\to Y$ with
$S_\sharp \mu=\nu$. When the transport condition $T_\sharp
\mu=\nu$ is satisfied, we say that $T$ is a \textit{transport
map}, and if $T$ minimizes also the cost we call it an
\textit{optimal transport map}. Up to now the optimal transport problem has been intensively
studied in a Euclidean or a Riemannian setting by many authors,
and it turns out that the particular choice $c(x,y)=d^2(x,y)$
(here $d$ denotes a Riemannian distance) is suitable for studying
some partial differential equations (like the semi-geostrophic
or porous medium equations), for
studying functional inequalities (like Sobolev and Poincar\'e-type
inequalities) and for applications in geometry (for example, in
the study of lower bound on the Ricci curvature of the manifolds).
We refer to the books \cite{agsbook,villani03, villaniSF} for
an excellent presentation.

After the existence and uniqueness results of Brenier for the
Euclidean case \cite{brenier91} and McCann for the Riemannian case
\cite{mccann01}, people tried to extend the theory in a
sub-Riemannian setting. In \cite{ar04} Ambrosio and Rigot studied
the optimal transport problem in the Heisenberg group, and
recently Agrachev and Lee were able to extend their result to
more general situations such as sub-Riemannian structures corresponding to $2$-generating distributions \cite{al07}.

Two key properties of the optimal transport map result to be
useful for many applications: the first one is the fact that the
transport map is differentiable a.e. (this for example allows to
write the Jacobian of the transport map a.e.), and the second one
is that, if $\mu$ and $\nu$ are absolutely continuous with respect
to the volume measure, so are all the measures belonging to the
(unique) Wasserstein geodesic between them. Both these properties
are true in the Euclidean case (see for example \cite{agsbook}) or
on compact Riemannian manifolds (see \cite{cems01,berbuf}). If the
manifold is noncompact, the second property still remains true
(see \cite[Section 5]{fatfig}), while the first one holds in a
weaker form.  Indeed, although one cannot hope for its
differentiability in the non-compact case, as proved in
\cite[Section 3]{figalliSIAM} (see also \cite{agsbook}) the
transport map is approximately differentiable a.e., and this turns
out to be enough for
extending many results from the compact to the non-compact case. Up to now, the only available results in these directions in a sub-Riemannian setting were proved in \cite{figjul}, where the authors show that the absolute continuity property along Wassertein geodesics holds in the Heisenberg group.\\

The aim of this paper is twofold: on the one hand, we prove
new existence and uniqueness results for the optimal transport map
on sub-Riemannian manifolds. In particular, we show that  the structure of the optimal transport map is more or less the same as in the Riemannian case (see \cite{mccann01}).  On the other hand, in a still large
class of cases, we prove that the transport map is (approximately)
differentiable almost everywhere, and that the absolute continuity
property along Wasserstein geodesics holds. This settles several open problems raised in \cite[Section
7]{ar04}: first of all, regarding problem \cite[Section 7 (a)]{ar04}, we are able to extend the results of
Ambrosio and Rigot \cite{ar04} and of Agrachev and Lee \cite{al07}
to a large class of sub-Riemannian manifolds, not necessarily two-generating. Concerning question
\cite[Section 7 (b)]{ar04}, we can prove a regularity result on optimal
transport maps, showing that under appropriate assumptions (including the Heisenberg group) they are approximately
differentiable a.e. Moreover,  under some weak assumptions on
the measures, the transport map is shown to be truly differentiable a.e. (see Theorem \ref{thmapproxdiff} and Remark
\ref{rmkdiff}). This allows for the first time in this setting to
apply the area formula, and to write a weak formulation of the
Monge-Amp\`ere equation (see Remark \ref{rmkMA}). Finally, Theorem
\ref{thmabscont} answers to problem \cite[Section 7 (c)]{ar04} not
only in the Heisenberg group (which was already solved in
\cite{figjul}) but also in more general cases.\\

The structure of the paper is the following:

In Section \ref{sect:prelimresults}, we introduce some concepts of sub-Riemannian geometry and optimal transport appearing in the statements of the results.

In Section \ref{sect:statementresults}, we present our results on
the mass transportation problem in sub-Riemannian geometry:
existence and uniqueness theorems on optimal transport maps
(Theorems \ref{mainthm} and \ref{thmrectif}), absolute continuity
property along Wasserstein geodesics (Theorem \ref{thmabscont}),
and finally regularity of the optimal transport map and its
consequences (Theorem \ref{thmapproxdiff} and Remarks
\ref{rmkdiff}, \ref{rmkMA}). For sake of simplicity, all the
measures appearing in these results are assumed to have compact
supports. In the last paragraph of Section
\ref{sect:statementresults} we discuss the possible extensions of
our results to the non-compact case.

In Section \ref{sect:examples}, we give a list of sub-Riemannian structures for which our different results may be applied. These cases include
fat distributions, two-generating distributions, generic distribution of rank $\geq 3$, nonholonomic distributions on three-dimensional manifolds, medium-fat distributions, codimension-one nonholonomic distributions, and rank-two distributions in four-dimensional manifolds.

Since the proofs of the theorems require lots of tools and results from sub-Riemannian geometry, we recall in Section \ref{sect:subRiemgeom} basic facts in sub-Riemannian geometry, such as the characterization of singular horizontal paths, the description of sub-Riemannian minimizing geodesics, or the properties of the sub-Riemannian  exponential mapping. Then, we present some results concerning the regularity of the sub-Riemannian distance function and its cut locus. These latter results are the key tools in the proofs of the our transport theorems.

In Section \ref{sect:proof}, taking advantage of the regularity properties obtained in the previous section,  we provide all the proofs of the results stated in Section \ref{sect:statementresults}.

Finally, in Appendix A, we recall some classical facts on
semiconcave functions, while in Appendix B we prove auxiliary
results needed in Section \ref{sect:examples}.

\section{Preliminaries}\label{sect:prelimresults}

\subsection{Sub-Riemannian manifolds}

A sub-Riemannian manifold is given by a triple $(M,\Delta,g)$
where $M$ denotes a smooth connected manifold of dimension $n$,
$\Delta$ is a smooth nonholonomic distribution of rank $m<n$ on
$M$, and $g$ is a Riemannian metric on $M$\footnote{Note that in
general the definition of a sub-Riemannian structure only involves
a Riemannian metric on the distribution. However, since in the
sequel we need a global Riemannian distance on the ambient
manifold and we need to use Hessians, we prefer to work with a
metric defined globally on $TM$.}. We recall that a smooth
distribution of rank $m$ on $M$ is a rank $m$ subbundle of $TM$.
This means that, for every $x \in M$, there exist a neighborhood
$\mathcal{V}_{x}$ of $x$ in $M$, and a $m$-tuple $(f_{1}^x,
\ldots, f_{m}^x)$ of smooth vector fields on $\mathcal{V}_{x}$,
linearly independent on $\mathcal{V}_{x}$, such that
$$
\Delta(z) =\SPAN
\left\{f_1^x(z),\ldots, f_m^x(z) \right\} \qquad \forall z\in \mathcal{V}_x.
$$
One says that the $m$-tuple of vector fields $(f_{1}^x,\ldots ,f_{m}^x)$
represents locally the distribution $\Delta$. The distribution $\Delta$ is said to be \textit{nonholonomic} (also called totally nonholonomic \textit{e.g.} in \cite{as04}) if, for every $x \in M$,
there is a $m$-tuple $(f_{1}^x, \ldots, f_{m}^x)$ of smooth vector fields on $\mathcal{V}_{x}$ which represents locally the distribution and such that
$$
\mbox{Lie} \left\{ f_1^x , \ldots, f_m^x \right\} (z) = T_zM \qquad \forall z \in \mathcal{V}_x,
$$
that is such that the Lie algebra\footnote{We recall that, for any
family $\mathcal{F}$ of smooth vector fields on $M$, the Lie
algebra of vector fields generated by $\mathcal{F}$, denoted by
$\mbox{Lie}(\mathcal{F})$, is the smallest vector space $S$
satisfying
$$
[X,Y] \subset S \qquad \forall X\in \mathcal{F}, \quad \forall Y \in S,
$$
where $[X,Y]$ is the Lie bracket of $X$ and $Y$.} spanned by $f_{1}^x, \ldots, f_{m}^x$, is equal to the whole tangent space $T_z M$ at every point $z\in \mathcal{V}_{x}$. This Lie algebra property is often called \textit{H\"ormander's condition}. \\

A curve $\gamma :[0,1] \rightarrow M$ is called a \textit{horizontal path} with respect to $\Delta$ if it belongs to $W^{1,2}([0,1],M)$ and satisfies
$$
\dot{\gamma}(t)\in\Delta(\gamma(t)) \qquad \mbox{for a.e. } t\in [0,1].
$$
According to the classical Chow-Rashevsky Theorem (see
\cite{bellaiche96,chow39,montgomery02,rashevsky38,riffordbook}),
since the distribution is nonholonomic on $M$, any two points of
$M$ can be joined by a horizontal path. That is for every $x,y \in
M$ there exists a horizontal path $\gamma :[0,1] \rightarrow M$
such that $\gamma(0)=x$ and $\gamma(1)=y$. For $x\in M$, let
$\Omega_\Delta(x)$ denote the set of horizontal paths
$\gamma:[0,1]\rightarrow M$ such that $\gamma(0)=x$. The set
$\Omega_\Delta(x)$, endowed with the $W^{1,2}$-topology, inherits
a Hilbert manifold structure (see \cite{montgomery02}). The
end-point mapping from $x$ is defined by
$$
\begin{array}{rcl}
\textrm{E}_{x} : \Omega_\Delta(x) & \longrightarrow & M\\
\gamma & \longmapsto & \gamma(1).
\end{array}
$$
It is a smooth mapping. A path $\gamma $ is said to be
\textit{singular} if it is horizontal and if it is a critical
point for the end-point mapping $\textrm{E}_{x}$, that is if the
differential of $\textrm{E}_{x}$ at $\gamma$ is singular (i.e. not
onto). A horizontal path which is not singular is called \textit{nonsingular} or \textit{regular}. Note that the regularity or singularity property of a given horizontal path depends only on the distribution, not on the metric $g$.\\

The \textit{length} of a path $\gamma \in
\Omega_{\Delta}(x)$ is defined by
\begin{equation}
\label{deflength}
\mbox{length}_{g}(\gamma ) := \int_0^1
\sqrt{g_{\gamma(t)}(\dot{\gamma}(t),\dot{\gamma}(t))} dt.
\end{equation}
The \textit{sub-Riemannian distance} $d_{SR}(x,y)$ (also called
Carnot-Carath\'eodory distance) between two points $x,y$ of $M$ is
the infimum over the lengths of the horizontal paths joining $x$
and $y$. Since the distribution is nonholonomic on $M$, according
to the Chow-Rashevsky Theorem (see
\cite{bellaiche96,chow39,montgomery02,rashevsky38,riffordbook})
the sub-Riemannian distance is finite and continuous\footnote{In
fact, thanks to the so-called Mitchell's ball-box Theorem (see
\cite{montgomery02}), the sub-Riemannian distance can be shown to
be locally H\"older continuous on $M \times M$.} on $M \times M$.
Moreover, if the manifold $M$ is a complete metric
space\footnote{Note that, since the distribution $\Delta$ is
nonholonomic on $M$, the topology defined by the sub-Riemannian
distance $d_{SR}$ coincides with the original topology of $M$ (see
\cite{bellaiche96,montgomery02}). Moreover, it can be shown that,
if the Riemannian manifold $(M,g)$ is complete, then for any
nonholonomic distribution $\Delta$ on $M$ the sub-Riemannian
manifold $(M,\Delta,g)$ equipped with its sub-Riemannian distance
is complete.} for the sub-Riemannian distance $d_{SR}$, then,
since $M$ is connected, for every pair $x,y$ of points of $M$
there exists a horizontal path $\gamma $ joining $x$ to $y$ such
that
$$
d_{SR}(x,y) = \mbox{length}_{g}(\gamma ).
$$
Such a horizontal path is called a \textit{sub-Riemannian minimizing geodesic} between $x$ and $y$. \\

Assuming that $(M,d_{SR})$ is complete, denote by $T^*M$ the cotangent bundle of $M$, by $\omega$ the canonical symplectic form on $T^*M$, and by $\pi:T^*M \rightarrow M$ the canonical projection.  The \textit{sub-Riemannian Hamiltonian} $H:T^*M\rightarrow\R$ which is canonically associated with the sub-Riemannian structure is defined as follows: for every $x\in M$,
the restriction of $H$ to the
fiber $T^*_xM$ is given by the nonnegative quadratic form
\begin{equation}
\label{monHam}
p \longmapsto \frac{1}{2}\max\left\{ \frac{p(v)^2}{g_x(v,v)}\
\vert\ v\in \Delta(x)\setminus\{0\} \right\} .
\end{equation}
Let $\overrightarrow{H}$ denote the Hamiltonian vector field on
$T^*M$ associated to $H$, that is
$\iota_{\overrightarrow{H}}\omega=-dH$. A \textit{normal extremal}
is an integral curve of $\overrightarrow{H}$ defined on $[0,1]$,
i.e. a curve $\psi(\cdot):[0,1]\rightarrow T^*M$ satisfying
$$
\dot{\psi}(t) = \overrightarrow{H}(\psi(t)), \qquad \forall t\in[0,1].
$$
Note that the projection of a normal extremal is a horizontal path with respect to $\Delta$. For every $x\in M$, the \textit{exponential mapping} with respect to $x$ is defined by
$$
\begin{array}{rcl}
\exp_{x} : T^*_xM & \longrightarrow & M\\
p & \longmapsto & \pi(\psi(1)),
\end{array}
$$
where $\psi$ is the normal extremal such that $\psi(0)=(x,p)$ in local coordinates. We stress that, unlike the Riemannian setting, the sub-Riemannian exponential mapping with respect to $x$ is defined on the cotangent space at $x$. \\

\noindent \textbf{Remark:} from now on, all sub-Riemannian manifolds appearing in the paper are assumed to be complete with respect to the sub-Riemannian distance.

\subsection{Preliminaries in optimal transport theory}
As we already said in the introduction, we recall that, given a cost function $c:X\times Y \to \R$, we are looking for a transport map
$T:X\to Y$ which minimizes the transportation cost $\int c(x,T(x))\,d\mu$.
The constraint $T_\#\mu=\nu$  being highly non-linear, the optimal transport problem is quite difficult from the viewpoint of calculus
of variation. The major advance on this problem was due to
Kantorovich, who proposed in \cite{kant1, kant2} a notion of weak
solution of the optimal transport problem. He suggested to look
for \textit{plans} instead of transport maps, that is probability
measures $\gamma$ in $X \times Y$ whose marginals are $\mu$ and
$\nu$, i.e.
$$
(\pi_X)_\sharp \gamma=\mu \qquad \text{and} \qquad (\pi_Y)_\sharp
\gamma=\nu,
$$
where $\pi_X: X \times Y \rightarrow X$ and $\pi_Y: X \times Y
\rightarrow Y$ are the canonical projections. Denoting by
$\Pi(\mu,\nu)$ the set of plans, the new minimization problem
becomes the following:
\begin{equation}
\label{kantprob0} C(\mu,\nu) =\min_{\gamma \in \Pi(\mu,\nu)}
\left\{ \int_{M \times M} c(x,y)\,d\gamma(x,y) \right\}.
\end{equation}
If $\gamma$ is a minimizer for the Kantorovich formulation, we say
that it is an \textit{optimal plan}. Due to the linearity of the
constraint $\gamma \in \Pi(\mu,\nu)$, it is simple using weak
topologies to prove existence of solutions to (\ref{kantprob0}):
this happens for instance whenever $X$ and $Y$ are Polish spaces,
and $c$ is lower semicontinuous and bounded from below (see for
instance \cite{villani03, villaniSF}). The connection between the
formulation of Kantorovich and that of Monge can be seen by
noticing that any transport map $T$ induces the plan defined by
$(Id \times T)_\sharp \mu$, which is concentrated on the graph of
$T$. Hence the problem of showing existence of optimal transport
maps can be reduced to prove that an optimal transport plan is
concentrated on a graph. Moreover, if one can show that
\textit{any} optimal plan in concentrated on a graph, since
$\frac{\g_1+\g_2}{2}$ is optimal if so are $\g_1$ and $\g_2$,
uniqueness of the transport map easily follows.

\begin{definition}{\rm
A function $\phi:X \to \R$ is said \textit{$c$-concave} if there exists a function
$\phi^c:Y \to \R \cup \{-\infty\}$, with $\phi^c \not\equiv
-\infty$, such that
$$
\phi(x)=\inf_{y \in Y} \left\{c(x,y) - \phi^c(y)\right\}.
$$
If $\phi$ is $c$-concave, we define the
\textit{$c$-superdifferential of $\phi$ at $x$} as
$$
\p^c\phi(x):=\{y \in Y \mid \phi(x) + \phi^c(y)=c(x,y)\}.
$$
Moreover we define the \textit{$c$-superdifferential of $\phi$} as
$$
\p^c\phi:=\{(x,y) \in X \times Y \mid y \in \p^c\phi(x)\}.
$$}
\end{definition}

As we already said in the introduction, we are interested in studying the optimal
transport problem on $M \times M$ ($M$ being a complete sub-Riemannian
manifold) with the cost function given by $c(x,y)=d_{SR}^2(x,y)$.

\begin{definition}{\rm
Denote by $P_c(M)$ the set of compactly supported probability
measures in $M$ and by $P_2(M)$ the set of Borel probability measures
on $M$ with finite $2$-order moment, that is the set of $\mu$ satisfying
$$
\int_M d_{SR}^2(x,x_0)\,d\mu(x)<+\infty \qquad \text{for some $x_0
\in M$.}
$$
Furthermore, we denote by $P_c^{ac}(M)$ (resp. $P^{ac}_2(M)$) the
subset of $P_c(M)$ (resp. $P_2(M)$) that consists of the
probability measures on $M$ which are absolutely continuous with
respect to the volume measure.}
\end{definition}

Obviously $P_c(M) \subset P_2(M)$. Moreover we remark that, by the
triangle inequality for $d_{SR}$, the definition of $P_2(M)$ does
not depend on $x_0$. The space $P_2(M)$ can be endowed with the
so-called \textit{Wasserstein distance} $W_2$:
$$
W_2^2(\mu,\nu):=\min_{\gamma \in \Pi(\mu,\nu)} \left\{ \int_{M
\times M} d^2(x,y)\,d\gamma(x,y) \right\}
$$
(note that $W_2^2$ is nothing else than the infimum in the
Kantorovich problem). As $W_2$ defines a finite metric on $P_2(M)$, one can speak about
geodesic in the metric space $(P_2,W_2)$. This space turns out,
indeed, to be a length space (see for example \cite{agsbook,villani03,
villaniSF}).\\

From now on, $\supp(\mu)$ and $\supp(\nu)$ will denote the supports of $\mu$ and $\nu$ respectively, i.e.
the smallest closed sets on which $\mu$ and $\nu$ are respectively concentrated.

The following result is well-known (see for instance \cite[Chapter
5]{villaniSF}):

\begin{theorem}
\label{dualitythm} Let us assume that $\mu,\nu \in P_2(M)$. Then
there exists a $c$-concave function $\phi$ such that the following
holds: a transport plan $\g \in \Pi(\mu,\nu)$ is optimal if and
only if $\g(\p^c \phi)=1$ (that is $\g$ is concentrated on the
$c$-superdifferential of $\phi$). Moreover one can assume that the
following holds:
$$
\phi(x)=\inf_{y \in \supp(\nu)} \left\{d_{SR}^2(x,y) - \phi^c(y)\right\} \qquad
\forall x \in M,
$$
$$
\phi^c(y)=\inf_{x \in \supp(\mu)} \left\{d_{SR}^2(x,y) - \phi(x)\right\} \qquad
\forall y \in M.
$$
In addition, if $\mu,\nu \in P_c(M)$, then both infima are indeed minima
(so that $\p^c \phi(x) \cap \supp(\nu)\neq \emptyset$ for $\mu$-a.e. $x$), and the functions
$\phi$ and $\phi^c$ are continuous.
\end{theorem}

By the above theorem we see that, in order to prove existence and
uniqueness of optimal transport maps, it suffices to prove that
there exist two Borel sets $Z_1,Z_2 \subset M$, with
$\mu(Z_1)=\nu(Z_2)=1$, such that $\p^c\phi$ is a graph inside $Z_1
\times Z_2$ (or equivalently that $\p^c\phi(x) \cap Z_2$ is a
singleton for all $x \in Z_1$).

\section{Statement of the results}\label{sect:statementresults}

\subsection{Sub-Riemannian versions of Brenier-McCann's Theorems}

The main difficulty appearing in the sub-Riemannian setting (unlike the Riemannian situation) is that, in general, the
squared distance function is not locally Lipschitz on the diagonal.
This gives rise to difficulties which make the proofs more technical than in the Riemannian case (and some new ideas are also needed).
In order to avoid technicalities
which would obscure the main ideas of the proof,
we will state our
results under some simplifying assumptions on the measures, and in Paragraph
\ref{subsect:noncpt} we will explain how to remove them.

Before stating our first existence and uniqueness result, we
introduce the following definition:
\begin{definition}{\rm
Given a $c$-concave function $\phi:M \to \R$, we define the
``moving" set $\Mov^\phi$ and the ``static" set $\Stat^\phi$ as
$$
\Mov^\phi:=\{x \in M \mid x \not\in\p^c\phi(x)\},
$$
$$
\Stat^\phi:=M\setminus\Mov^\phi=\{x \in M \mid x\in\p^c\phi(x)\}.
$$
}
\end{definition}

We will also denote by $\pi_1:M\times M \to M$ and $\pi_2:M\times M \to M$ the canonical projection on the first
and on the second factor, respectively. In the sequel, $D$ denotes the diagonal in $M\times M$, that is
$$
D  := \left\{ (x,y)\in M \times M \mid x=y\right\}.
$$
Furthermore, we refer the reader to Appendix A for the definition of a locally semiconcave function.

\begin{theorem}[\bf Optimal transport map for absolutely continuous measures]
\label{mainthm} Let $\mu \in P_c^{ac}(M)$, $\nu \in P_c(M)$. Assume that there
exists an open set $\Omega \subset M\times M$ such that
$\supp(\mu\times \nu) \subset \O$, and $d_{SR}^2$ is locally
semiconcave (resp. locally Lipschitz) on $\Omega \setminus D$.
Let $\phi$ be the $c$-concave function provided by Theorem \ref{dualitythm}. Then:
\begin{enumerate}
\item[(i)] $\Mov^\phi$ is open, and $\phi$ is locally semiconcave (resp. locally Lipschitz) in a neighborhood of $\Mov^\phi\cap\supp(\mu)$. In particular $\phi$ is
differentiable $\mu$-a.e. in $\Mov^\phi$.
\item[(ii)] For $\mu$-a.e.
$x \in \Stat^\phi$, $\p^c\phi(x)=\{x\}$.
\end{enumerate}
In particular, there exists a unique optimal transport map
defined $\mu$-a.e. by\footnote{The factor $\frac{1}{2}$ appearing in front of $d\phi(x)$ is due to the fact that we are considering the cost function $d_{SR}^2(x,y)$ instead of the (equivalent)
cost $\frac{1}{2}d_{SR}^2(x,y)$}
$$
T(x):= \left\{
\begin{array}{cc}
\exp_x(-\textstyle{\frac{1}{2}}\,d\phi(x))& \text{if } x\in\Mov^\phi\cap\supp(\mu),\\
x& \text{if }x \in\Stat^\phi\cap\supp(\mu),
\end{array}
\right.
$$
and for $\mu$-a.e. $x$ there exists a unique minimizing geodesic between $x$ and $T(x)$.
\end{theorem}

The two main issues in the proof of the above theorem are the regularity of the $c$-concave function $\phi$ provided by Theorem \ref{dualitythm} and the existence and uniqueness of minimizing projections of normal extremals between almost all pairs of points in $\p^c\phi$. Roughly speaking, the regularity properties of $\phi$ are consequences of regularity assumptions made on the cost function while the second issue is tackled (as it was already done by Agrachev and Lee in \cite{al07}) by transforming  a problem with end-point constraint into a problem with free end-point (see Proposition \ref{propexp1}).  Furthermore, as can be seen from the proof (given in
Section \ref{sect:proof}),  assertion (ii) in Theorem
\ref{mainthm} always holds without any assumption on the
sub-Riemannian distance. That is, for any optimal transport
problem on a complete sub-Riemannian manifold between two measures
$\mu \in P_c^{ac}(M)$ and $\nu \in P_c(M)$, we always have
$$
\p^c\phi(x) = \{x\} \qquad \mbox{for }\mu\mbox{-a.e. } x\in
\Stat^\phi,
$$
where $\phi$ is the $c$-concave function provided by Theorem
\ref{dualitythm}. Such a result is a consequence of a Pansu-Rademacher Theorem which was already used by Ambrosio and Rigot in \cite{ar04}.\\

Theorem \ref{mainthm} above can be refined if the sub-Riemannian
distance is assumed to be locally Lipschitz on the diagonal. In
that way, we obtain the sub-Riemannian version of McCann's Theorem
on Riemannian manifolds (see \cite{mccann01}), improving the result of Agrachev and Lee (see \cite{al07}).\\

\begin{theorem}[\bf Optimal transport map for more general measures]
\label{thmrectif} Let $\mu,\nu \in P_c(M)$, and
suppose that $\mu$ gives no measure to countably $(n-1)$-rectifiable sets.
Assume that there
exists an open set $\Omega \subset M\times M$ such that
$\supp(\mu\times \nu) \subset \O$, and $d_{SR}^2$ is locally
semiconcave on $\Omega \setminus D$. Suppose further that $d_{SR}^2$ is locally
Lipschitz on $\Omega$, and let $\phi$ be the $c$-concave function provided by Theorem \ref{dualitythm}.
Then:
\begin{enumerate}
\item[(i)] $\Mov^\phi$ is open, and $\phi$ is locally semiconcave in a neighborhood of $\Mov^\phi\cap\supp(\mu)$.
In particular $\phi$ is differentiable $\mu$-a.e. in $\Mov^\phi$.
\item[(ii)] For $\mu$-a.e.
$x \in \Stat^\phi$, $\p^c\phi(x)=\{x\}$.
\end{enumerate}
In particular, there exists a unique optimal transport map
defined $\mu$-a.e. by
$$
T(x):= \left\{
\begin{array}{cc}
\exp_x(-\textstyle{\frac{1}{2}}\,d\phi(x))& \text{if } x\in\Mov^\phi\cap\supp(\mu),\\
x& \text{if }x \in\Stat^\phi\cap\supp(\mu),
\end{array}
\right.
$$
and for $\mu$-a.e. $x$ there exists a unique minimizing geodesic between $x$ and $T(x)$.
\end{theorem}

The regularity properties of the sub-Riemannian distance functions
required in the two results above are satisfied by many
sub-Riemannian manifolds. In particular Theorem \ref{mainthm}
holds as soon as there are no singular sub-Riemannian minimizing
geodesic between two distinct points in $\Omega$.  In Section
\ref{sect:examples}, we provide a list of sub-Riemannian manifolds
which satisfy the assumptions of our different results.

\subsection{Wasserstein geodesics}

Thanks to Theorem \ref{mainthm}, it is not difficult to deduce the
uniqueness of the Wasserstein geodesic between $\mu$ and $\nu$.
Moreover the structure of the transport map allows to prove, as in
the Riemannian case, that all the measures inside the geodesic are
absolutely continuous if $\mu$ is. This last property requires
however that, if $(x,y)\in \O$, then all geodesics from $x$ to $y$
do not ``exit from $\O$":

\begin{definition}{\rm
Let $\O\subset M\times M$ be an open set. We say that $\O$ is \textit{totally geodesically convex}
if for every $(x,y) \in \O$ and every geodesic $\g:[0,1]\to M$ from $x$ to $y$, one has
$$
(x,\g(t)), \, (\g(t),y) \in \O \qquad \forall t \in [0,1].
$$
}
\end{definition}

Observe that, if $\O=U\times U$ with $U \subset M$, then the above definition reduces to
say that $U$ is totally geodesically convex in the classical sense.

\begin{theorem}[\bf Absolute continuity of Wasserstein geodesics]
\label{thmabscont} Let $\mu \in P_c^{ac}(M)$, $\nu \in P_c(M)$.
Assume that there exists an open set $\O \subset M\times M$ such that
$\supp(\mu\times\nu) \subset \O$, and $d_{SR}^2$ is locally
semiconcave on $\O \setminus D$. Let $\phi$ be the $c$-concave function provided by Theorem \ref{dualitythm}.
Then there exists a unique
Wasserstein geodesic $(\mu_t)_{t \in [0,1]}$ joining $\mu=\mu_0$ to
$\nu=\mu_1$, which is given by $\mu_t:=(T_t)_\#\mu$ for $t \in [0,1]$, with
$$
T_t(x):= \left\{
\begin{array}{cc}
\exp_x(-\textstyle{\frac{t}{2}}\,d\phi(x))& \text{if } x\in\Mov^\phi\cap \supp(\mu),\\
x& \text{if }x \in\Stat^\phi\cap \supp(\mu).
\end{array}
\right.
$$
Moreover, if $\O$ is totally geodesically convex, then $\mu_t \in
P_c^{ac}(M)$ for all $t \in [0,1)$.
\end{theorem}

\subsection{Regularity of the transport map and the Monge-Amp\`ere equation}

The structure of the transport map provided by Theorem
\ref{mainthm} allows also to prove in certain cases the approximate
differentiability of the optimal transport map, and a useful
Jacobian identity. Let us first recall the notion of approximate differential:

\begin{definition}[\bf Approximate differential]\label{approxdiff}{\rm
We say that $f :M \rightarrow \R$ has an \textit{approximate
differential} at $x \in M$ if there exists a function $h:M
\rightarrow \R$ differentiable at $x$ such that the set $\{f =
h\}$ has density $1$ at $x$ with respect to the volume measure. In
this case, the approximate value of $f$ at $x$ is defined as
$\tilde f(x)=h(x)$, and the approximate differential of $f$ at $x$
is defined as $\tilde df(x)=dh(x)$.
}
\end{definition}

It is not difficult to show that the above definitions make sense. In
fact, $h(x)$ and $dh(x)$ do not depend on the choice of
$h$, provided $x$ is a density point of the set $\{f = h\}$.\\

To write the formula of the Jacobian of $T$, we will need to use the notion of Hessian.
We recall that the Hessian of a function $f :M \rightarrow \R$ is defined as the covariant derivative
of $df$: $\Hess f(x)=\nabla df(x):T_x M \times T_x M \to M$.
Observe that the notion of the Hessian depends on the Riemannian metric on $TM$. However, since the transport map
depends only on $d_{SR}$, which in turn depends only on the restriction of metric to the distribution,
a priori it may seem strange that the Jacobian of $T$ is expressed in terms of Hessians.
However, as we will see below, the Jacobian of $T$ depends on the Hessian of the function
$z \mapsto \phi(z)-d_{SR}^2(z,T(x))$ computed at $z=x$. But since $\phi(z)-d_{SR}^2(z,T(x))$
attains a maximum at $x$, $x$ is a critical point for the above function, and so its Hessian at $x$
is indeed independent on the choice of the metric.

The following result is the sub-Riemannian version of the
properties of the transport map in the Riemannian case. It was
proved on compact manifolds in \cite{cems01}, and extended to the
noncompact case in \cite{figalliSIAM}. The main difficulty in our
case comes from the fact that the structure of the sub-Riemannian
cut-locus is different with respect to the Riemannian case, and so
many complications arise when one tries to generalize the
Riemannian argument to our setting. Trying to extend the
differentiability of the transport map in great generality would
need some new results on the sub-Riemannian cut-locus which go
behind the scope of this paper (see the Open Problem in Paragraph
\ref{para:cut}). For this reason, we prefer to state the result
under some simplifying assumptions, which however holds in the
important case of the Heisenberg group (see \cite{montgomery02}),
or for example for the standard sub-Riemannian structure on the
three-sphere (see \cite{br07}).

We refer the reader to Paragraph \ref{para:cut} for the definitions of the global cut-locus $\mbox{Cut}_{SR} (M)$.

\begin{theorem}[\bf Approximate differentiability and jacobian identity]
\label{thmapproxdiff} Let $\mu\in P_c^{ac}(M)$, $\nu \in P_c(M)$.
Assume that there exists a totally geodesically convex  open set
$\Omega \subset M\times M$ such that $\supp(\mu\times \nu) \subset
\O$, $d_{SR}^2$ is locally semiconcave on $\Omega \setminus D$,
and for every $(x,y)\in \mbox{Cut}_{SR} (M) \cap \left( \Omega
\setminus  D\right)$ there are at least two distinct
sub-Riemannian minimizing geodesics joining $x$ to $y$. Let $\phi$
be the $c$-concave function provided by Theorem \ref{dualitythm}.
Then the optimal transport map is differentiable for $\mu$-a.e. $x
\in\Mov^\phi\cap\supp(\mu)$, and it is approximately
differentiable $\mu$-a.e. Moreover
$$
Y(x):=d(\exp_x)_{-\textstyle{\frac{1}{2}}d\phi(x)} \quad
\mbox{and} \quad  H(x):=\frac{1}{2}\Hess d_{SR}^2(\cdot,T(x))|_{z=x}
$$ exists for $\mu$-a.e.
$x \in\Mov^\phi\cap\supp(\mu)$, and the approximate differential of $T$ is
given by the formula
$$
\tilde dT(x)= \left\{
\begin{array}{cc}
Y(x) \bigl(H(x)-\textstyle{\frac{1}{2}}\Hess\phi(x) \bigr)& \text{if } x\in\Mov^\phi\cap\supp(\mu),\\
Id & \text{if } x\in\Stat^\phi\cap\supp(\mu),
\end{array}
\right.
$$
where $Id:T_x M \rightarrow T_x M$ denotes the identity map.

Finally, assuming both $\mu$ and $\nu$ absolutely continuous with respect to the volume measure, and denoting
by $f$ and $g$ their respective density, the
following Jacobian identity holds:
\begin{equation}
\label{eqjacident} \det\bigl(\tilde d
T(x)\bigr) =\frac{f(x)}{g(T(x))}\neq 0 \qquad \text{$\mu$-a.e.}
\end{equation}
In particular, $f(x)=g(x)$ for $\mu$-a.e. $x\in\Stat^\phi\cap\supp(\mu)$.
\end{theorem}

\begin{remark}[\bf Differentiability a.e. of the transport map]
\label{rmkdiff}{\rm If we assume that $f\neq g$ $\mu$-a.e., then
by the above theorem we deduce that $T(x)\neq x$ $\mu$-a.e. (or
equivalently $x \not\in \p^c\phi(x)$ $\mu$-a.e.).
Therefore the optimal transport is given by
$$
T(x)=\exp_x(-\textstyle{\frac{1}{2}}d\phi(x)) \qquad \text{$\mu$-a.e.},
$$
and in particular $T$ is differentiable (and not only approximate
differentiable) $\mu$-a.e.}
\end{remark}

\begin{remark}[\bf The Monge-Amp\`ere equation]\label{rmkMA}
{\rm Since the function $z
\mapsto \phi(z)-d_{SR}^2(z,T(x))$ attains a maximum at $T(x)$ for $\mu$-a.e. $x$, it
is not difficult to see that the matrix
$H(x)-\frac{1}{2}\Hess\phi(x)$ (defined in Theorem \ref{thmapproxdiff}) is nonnegative definite $\mu$-a.e.
This fact, together with \eqref{eqjacident}, implies that the function $\phi$
satisfies the \textit{Monge-Amp\`ere} type equation
$$
\det\bigl(H(x)-\textstyle{\frac{1}{2}}\Hess\phi(x)
\bigr)=\frac{f(x)}{|\det(Y(x))|g(T(x))}\qquad \text{for $\mu$-a.e. $x\in\Mov^\phi$.}
$$
In particular, thanks to Remark \ref{rmkdiff},
$$
\det\bigl(H(x)-\textstyle{\frac{1}{2}}\Hess\phi(x)
\bigr)=\frac{f(x)}{|\det(Y(x))|g(T(x))}\qquad \text{$\mu$-a.e.}
$$
provided that $f\neq g$ $\mu$-a.e.
}
\end{remark}

\subsection{The non-compact case}
\label{subsect:noncpt}

Let us briefly show how to remove the compactness assumption on
$\mu$ and $\nu$, and how to relax the hypothesis
$\supp(\mu\times\nu) \subset \O$. We assume $\mu,\nu \in P_2(M)$
(so that Theorem \ref{dualitythm} applies), and that $\mu\times
\nu(\O)=1$. Take an increasing sequence of compact sets $K_\ell
\subset \O$ such that $\cup_{\ell\in\N} K_\ell= \O$. We consider
$$
\psi_\ell(x) := \inf \left\{ d_{SR}^2(x,y) - \phi^c(y) \ \vert \ y
\mbox{ s.t. }  (x,y)\in K_\ell  \right\}.
$$
Since now $\phi^c$ is not a priori continuous (and so
$\p^c\psi_\ell$ is not necessarily closed), we first define
$$
\phi_\ell^c(y) := \inf \left\{ d_{SR}^2(x,y) - \psi_\ell(x) \
\vert \ x \mbox{ s.t. } (x,y)\in K_\ell \right\},
$$
and then consider
$$
\phi_\ell(x) := \inf  \left\{ d_{SR}^2(x,y) - \phi_\ell^c(y) \
\vert \ y \mbox{ s.t. } (x,y)\in K_\ell \right\}.
$$
In this way the following properties holds (see for example the
argument in the proof of \cite[Proposition 5.8]{villaniSF}):
\begin{enumerate}
\item[-] $\phi_\ell$ and $\phi_\ell^c$ are both continuous;
\item[-] $\psi_\ell(x) \geq \phi(x)$ for all $x\in M$; \item[-]
$\phi^c(y) \leq \phi_\ell^c(y)$ for all $y\in \pi_2(K_\ell)$;
\item[-] $\phi_\ell(x) =\psi_\ell(x)$ for all $x \in
\pi_1(K_\ell)$.
\end{enumerate}
This implies that $\p^c\phi \cap K_\ell \subset \p^c\phi_\ell$,
and so
$$
\p^c\phi\cap\O \subset \bigcup_{\ell\in\N} \p^c\phi_\ell.
$$
One can therefore prove (i) and (ii) in Theorem \ref{mainthm} with
$\phi_\ell$ in place of $\phi$, and from this and the hypothesis
$\mu\times \nu(\O)=1$ it is not difficult to deduce that
$(x,\p^c\phi(x)) \cap \Omega$ is a singleton for $\mu$-a.e. $x$
(see the argument in the proof of Theorem \ref{mainthm}). This
proves existence and uniqueness of the optimal transport map.

Although in this case we cannot hope for any semiconcavity result
for $\phi$ (since, as in the non-compact Riemannian case, $\phi$
is just a Borel function), the above argument shows that the graph
of the optimal transport map is contained in the union of
$\p^c\phi_\ell$. Hence, as in \cite[Section 5]{fatfig} one can use
$\p^c\phi_\ell$ to construct the (unique) Wasserstein geodesic
between $\mu$ and $\nu$, and in this way the absolutely continuity
of all measures belonging to the geodesic follows as in the
compactly supported case.

Finally, the fact that the graph of the optimal transport map is
contained in $\cup_{\ell\in\N}\p^c\phi_\ell$ allows also to prove
the approximate differentiability of the transport map and the
Jacobian identity, provided that one replaces the hessian of
$\phi$ with the approximate hessian (we refer the reader to
\cite[Section 3]{figalliSIAM} to see how this argument works in
the Riemannian case).

\section{Examples}
\label{sect:examples}

The aim of the present section is to provide a list of examples
where some of our theorems apply. For each kind of sub-Riemannian
manifold that we present, we provide a regularity result for the
associated squared sub-Riemannian distance function. We leave to
the reader to check in each case which of our theorems holds under
that regularity property.  Before giving examples, we recall that,
if $\Delta$ is a smooth distribution on $M$, a \emph{section} of
$\Delta$ is any smooth vector field $X$ satisfying $X(x) \in
\Delta (x)$ for any $x\in M$. For any smooth vector field  $Z$ on
$M$ and every $x\in M$, we shall denote by $[Z,\Delta] (x),
[\Delta,\Delta](x),$ and $[Z,[\Delta,\Delta]]$ the subspaces of
$T_xM$ given by
$$
[Z,\Delta] (x) := \left\{ [Z,X](x)  \mid X \mbox{ section of } \Delta \right\},
$$
$$
[\Delta, \Delta] (x) := \SPAN \left\{ [X,Y](x) \mid X,Y \mbox{ sections of } \Delta \right\},
$$
$$
[Z,[\Delta,\Delta]] (x) := \SPAN \left\{ [Z,[X,Y]](x) \mid X,Y \mbox{ sections of } \Delta \right\}.
$$

\subsection{Fat distributions}

The distribution $\Delta$ is called \textit{fat} if, for every $x\in M$ and every vector field $X$
on $M$ such that $X(x) \in \Delta(x)\setminus\{0\}$, there holds
$$
T_xM = \Delta(x)+ [X,\Delta](x).
$$
The above condition being very restrictive, there are very few fat
distributions (see \cite{montgomery02}). Fat distributions on
three-dimensional manifolds are the rank-two distributions
$\Delta$ satisfying
$$
T_x M = \SPAN \{f_1(x),f_2(x),[f_1,f_2](x)\} \qquad \forall x \in M,
$$
where $(f_1,f_2)$ is a $2$-tuple of vector fields representing locally the distribution $\Delta$. A classical example of fat distribution in $\R^3$ is given by the distribution spanned by the vector fields
$$
X_1 = \frac{\partial}{\partial x_1} , \qquad X_2 = \frac{\partial}{\partial x_2} + x_1\frac{\partial }{\partial x_3}.
$$
This is the distribution appearing in the Heisenberg group (see
\cite{ar04,bellaiche96,figjul}). It can be shown that, if $\Delta$
is a fat distribution, then any nontrivial (i.e. not constant)
horizontal path with respect to $\Delta$ is nonsingular (see
\cite{cr08,montgomery02,riffordbook}).  As a consequence, Theorems
\ref{THMsc} and \ref{THMlip} yield the following result:

\begin{proposition}
If $\Delta$ is fat on $M$, then the squared sub-Riemannian distance function is locally Lipschitz on $M\times M$ and locally semiconcave on $M\times M \setminus D$.
\end{proposition}

\subsection{Two-generating distributions}

A distribution $\Delta$ is called \textit{two-generating} if
$$
T_xM = \Delta(x) + [\Delta,\Delta] (x) \qquad \forall x\in M.
$$
Any fat distribution is two-generating. Moreover, if the ambient manifold $M$ has dimension three, then any two-generating distribution is fat. The distribution $\Delta$ in $\R^4$ which is spanned by the vector fields
$$
X_1 = \frac{\partial }{\partial x_1}, \qquad X_2 = \frac{\partial}{\partial x_2}, \qquad X_3 =\frac{\partial}{\partial x_3} + x_1\frac{\partial}{\partial x_4},
$$
provides an example of distribution which is two-generating but
not fat. It is easy to see that, if the distribution is
two-generating, then there are no Goh paths (see Paragraph
\ref{subsect:locLip} for the definition of Goh path). As a
consequence, by Theorem \ref{THMlip} we have:

\begin{proposition}
If $\Delta$ is two-generating on $M$, then the squared sub-Riemannian distance function is locally Lipschitz on $M\times M$.
\end{proposition}

The above result and its consequences in optimal transport are due to Agrachev and Lee (see \cite{al07}).

\subsection{Generic sub-Riemannian structures}

Let $(M,g)$ be a complete Riemannian manifold of dimension $\geq
4$, and  $m\geq 3$ be a positive integer.  Denote by ${\cal D}_m$
the space of rank $m$ distributions  on $M$ endowed with the
Whitney $C^\infty$ topology. Chitour, Jean and Tr\'elat proved
that there exists an open dense subset $\mathcal{O}_m$ of ${\cal
D}_m$ such that every element of $\mathcal{O}_m$ does not admit
nontrivial minimizing singular paths (see \cite{cjt06,cjt}). As a
consequence, we have:

\begin{proposition}
Let $(M,g)$ be a complete Riemannian manifold of dimension $\geq
4$. Then, for any generic distribution of rank $\geq 3$, the
squared sub-Riemannian distance function is locally semiconcave on
$M\times M\setminus D$.
\end{proposition}

This result implies in particular that, for generic sub-Riemannian manifolds, we have existence and uniqueness of optimal transport maps,
and absolute continuity of Wasserstein geodesics.

\subsection{Nonholonomic distributions on three-dimensional manifolds}

Assume that $M$ has dimension $3$ and that $\Delta$ is a nonholonomic rank-two distribution on $M$, and define
$$
\Sigma_{\Delta} := \left\{ x \in M \mid \Delta (x) + [\Delta,\Delta] (x) \neq \R^3\right\}.
$$
The set $\Sigma_{\Delta}$ is called the \textit{singular set} or the \textit{Martinet set} of $\Delta$. As an example, take the nonholonomic distribution  $\Delta$ in $\R^3$ which is spanned by the vector fields
$$
f_1 = \frac{\partial}{\partial x_1} , \qquad f_2 = \frac{\partial}{\partial x_2} + x_1^2 \frac{\partial }{\partial x_3}.
$$
It is easy to show that the singular set of $\Delta$ is the plane $\{x_1=0\}$. This distribution is often called the \textit{Martinet distribution}, and $\Sigma_{\Delta}$ the \textit{Martinet surface}. The singular horizontal paths of $\Delta$ correspond to the horizontal paths which are included in $\Sigma_{\Delta}$. This means that necessarily any singular horizontal path is, up to reparameterization,
a restriction of an arc of the form $t\mapsto (0,t,\bar{x}_3) \in \R^3$ with $\bar{x}_3\in \R$. This kind of result holds for any rank-two distribution in dimension three (we postpone its proof to Appendix B):

\begin{proposition}
\label{propdim3} Let $\Delta$ be a nonholonomic distribution on a
three-dimensional manifold. Then $\Sigma_{\Delta}$ is a closed
subset of $M$ which is countably $2$-rectifiable. Moreover a
nontrivial horizontal path $\gamma :[0,1] \rightarrow M$ is
singular if and only if it is included in $\Sigma_{\Delta}$.
\end{proposition}

Proposition \ref{propdim3} implies that for any pair $(x,y) \in M\times M$ (with $x\neq y$) such that $x$ or $y$ does not belong to $\Sigma_{\Delta}$, any sub-Riemannian minimizing  geodesic between $x$ and $y$ is nonsingular. As a consequence, thanks to Theorems \ref{THMsc} and \ref{THMlip}, the following result holds:

\begin{proposition}
\label{propdim3_2}
Let $\Delta$ be a nonholonomic distribution on a three-dimensional manifold. The squared sub-Riemannian distance function is locally Lipschitz on $M\times M \setminus \left( \Sigma_{\Delta} \times \Sigma_{\Delta}\right)$ and locally semiconcave on $M\times M \setminus \left( D \cup \Sigma_{\Delta} \times \Sigma_{\Delta}\right)$.
\end{proposition}

We observe that, since $\Sigma_{\Delta}$ is countably $2$-rectifiable, for any pair of measures $\mu, \nu \in P_c(M)$ such that $\mu$ gives no measure to countably $2$-rectifiable sets, the conclusions of Theorem \ref{thmrectif} hold.

\subsection{Medium-fat distributions}

The distribution $\Delta$ is called \textit{medium-fat}  if, for every $x\in M$ and every vector field $X$
on $M$ such that $X(x)\in \Delta(x)\setminus\{0\}$, there holds
$$
T_xM = \Delta(x)+ [\Delta,\Delta](x) + [X,[\Delta,\Delta]](x).
$$
Any two-generating distribution is medium-fat. An example of medium-fat distribution which is not two-generating is given by the rank-three distribution in $\R^4$ which is spanned by the vector  vector fields
$$
f_1 = \frac{\partial}{\partial x_1} , \qquad f_2 = \frac{\partial}{\partial x_2}, \qquad f_3=  \frac{\partial }{\partial x_3} + (x_1+x_2+x_3)^2  \frac{\partial}{\partial x_4}.
$$
Medium-fat distribution were introduced by Agrachev and Sarychev in \cite{as99} (we refer the interested reader to that paper for a detailed study of this kind of distributions). It can easily be shown that medium-fat distributions do not admit nontrivial Goh paths. As a consequence, Theorem \ref{THMlip} yields:

\begin{proposition}
Assume that $\Delta$ is medium-fat. Then the squared sub-Riemannian distance function is locally Lipschitz on $M\times M \setminus D$.
\end{proposition}

Let us moreover observe that, given a medium-fat distribution, it
can be shown that for a generic smooth complete Riemannian metric
on $M$ the distribution does not admit nontrivial singular
sub-Riemannian  minimizing geodesics (see \cite{cjt06,cjt}). As a
consequence, we have:

\begin{proposition}
Let $\Delta$ be a medium-fat distribution on $M$. Then, for ``generic"  Riemannian metrics, the squared sub-Riemannian distance function is locally semiconcave on $M\times M \setminus D$.
 \end{proposition}

Notice that, since two-generating distributions are medium-fat, the latter result holds for two-generating distributions.

\subsection{Codimension-one nonholonomic distributions}

Let $M$ have dimension $n$, and $\Delta$ be a nonholonomic distribution of rank $n-1$.
As in the case of nonholonomic distributions on three-dimensional manifolds, we can define the singular set associated to the distribution as
$$
\Sigma_{\Delta} := \left\{ x \in M \mid \Delta(x) + [\Delta,\Delta](x) \neq T_x M \right\}.
$$
The following result holds (we postpone its proof to Appendix B):

\begin{proposition}
\label{propcodim1} If $\Delta$ is a nonholonomic distribution of
rank $n-1$, then the set $\Sigma_{\Delta}$ is a closed subset of
$M$ which is countably $(n-1)$-rectifiable. Moreover any Goh path
is contained in $\Sigma_{\Delta}$.
\end{proposition}

From Theorem \ref{THMlip}, we have:

\begin{proposition}
The squared sub-Riemannian distance function is locally Lipschitz on $M\times M \setminus \left( \Sigma_{\Delta} \times \Sigma_{\Delta}\right)$.
\end{proposition}

Note that, as for medium-fat distributions, for generic metrics the function $d_{SR}^2$ is locally semiconcave on $M\times M \setminus \left( D \cup \Sigma_{\Delta} \times \Sigma_{\Delta}\right)$.

\subsection{Rank-two distributions in dimension four}

Let $(M,\Delta,g)$ be a complete sub-Riemannian manifold of dimension four, and let $\Delta$ be a \textit{regular} rank-two distribution, that is
$$
T_x M = \SPAN \left\{f_1 (x), f_2 (x), [f_1,f_2](x), [f_1,[f_1,f_2]](x), [f_2,[f_1,f_2]](x)  \right\}
$$
for any local parametrization of the distribution. In
\cite{sussmann96} Sussmann shows  that there is a smooth
horizontal vector field $X$ on $M$ such that the singular
horizontal curves $\gamma$ parametrized by arc-length are exactly
the integral curves of $X$, i.e. the curves satisfying
$$
\dot{\gamma}(t) = X(\gamma(t)).
$$
By the way, it can also be shown that those curves are locally
minimizing between their end-points (see \cite{ls95,sussmann96}).
For every $x\in M$, denote by $\mathcal{O}(x)$ the orbit of $x$ by
the flow of $X$, and set
$$
\Omega := \left\{ (x,y) \in M \times M \mid y \notin \mathcal{O}(x)\right\}.
$$
Sussmann's Theorem, together with Theorem \ref{THMsc}, yields the following result:

\begin{proposition}
\label{propsuss} Under the above assumption, the function
$d_{SR}^2$ is locally semiconcave in the interior of $\Omega$.
\end{proposition}

As an example, consider the distribution $\Delta$ in $\R^4$ spanned by the two vector fields
$$
f_1 = \frac{\partial }{\partial x_1}, \qquad f_2= \frac{\partial}{\partial x_2} + x_1 \frac{\partial}{\partial x_3} + x_3 \frac{\partial}{\partial x_4}.
$$
It is easy to show that a horizontal path $\gamma :[0,1] \rightarrow \R^4$ is singular if and only if it satisfies, up to  reparameterization by arc-length,
$$
\dot{\gamma}(t) = f_1(\gamma(t)) \qquad \forall t\in [0,1].
$$
By Proposition \ref{propsuss} we deduce that, for any complete
metric $g$ on $\R^4$, the function $d_{SR}^2$ is locally
semiconcave on the set
$$
\Omega = \left\{ (x,y) \in \R^4 \times \R^4 \mid (y-x) \notin \SPAN \{e_1\} \right\},
$$
where $e_1$ denotes the first vector in the canonical basis of $\R^4$. Consequently, for any pair of measures $\mu \in P_c^{ac}(M)$, $\nu \in P_c(M)$ satisfying
$\supp(\mu\times \nu) \subset \O$, Theorem \ref{mainthm} applies (or more in general, if $\mu\times \nu(\O)=1$, we can apply the argument in
Paragraph \ref{subsect:noncpt}).

\section{Facts in sub-Riemannian geometry}
\label{sect:subRiemgeom}

Throughout this section $(M,\Delta,g)$ denotes a sub-Riemannian
manifold of rank $m<n$, which is assumed to be complete with
respect to the sub-Riemannian distance. As in the Riemannian case,
the Hopf-Rinow Theorem holds. In particular any two points in $M$
can be joined by a minimizing geodesics, and any sub-Riemannian
ball of finite radius is a compact subset of $M$. We refer the
reader to \cite[Appendix D] {montgomery02} for the proofs of those
results. We present in the following subsections  a list of basic
facts in sub-Riemannian geometry, whose the proofs may be found in
\cite{montgomery02} and \cite{riffordbook}.

\subsection{Nonholonomic distributions vs. nonholonomic control systems}

Any nonholonomic distribution can be locally parameterized by a
nonholonomic control system, that is by a smooth dynamical system
with parameters called controls. Indeed, assume that $\mathcal{V}$
is an open subset of $M$ such that there are $m$ smooth vector
fields $f_1,\ldots ,f_m$ on $\mathcal{V}$ which parametrize the
nonholonomic distribution $\Delta$ on $\mathcal{V}$, that is which
satisfy
$$
\Delta  (x) = \SPAN \left\{f_1 (x),\ldots, f_m(x) \right\} \qquad \forall x\in \mathcal{V},
$$
and
$$
\mbox{Lie} \left\{ f_1 , \ldots, f_m \right\} (x) = T_xM \qquad \forall x \in \mathcal{V}.
$$
Given $x\in \mathcal{V}$, there is a correspondence  between the set of horizontal paths in $\Omega_{\Delta}(x)$ which remain in $\mathcal{V}$ and the set of admissible controls of the control system
$$
\dot{x} = \sum_{i=1}^m u_i f_i(x).
$$
A \textit{control} $u\in L^2([0,1],\R^m)$ is called \textit{admissible} with respect to $x$ and $\mathcal{V}$ if the solution $\gamma_{x,u}$ to the Cauchy problem
$$
\dot{x}(t) = \sum_{i=1}^m u_i(t) f_i(x(t)) \qquad \mbox{for a.e. } t\in [0,1], \quad x(0)=x,
$$
is well-defined on $[0,1]$ and remains in $\mathcal{V}$. The set $\mathcal{U}_x$ of admissible controls is an open subset of $L^2 ([0,1],\R^m)$.

\begin{proposition}
Given $x\in M$, the mapping
$$
\begin{array}{rcl}
\mathcal{U}_x & \longrightarrow & \Omega_{\Delta}(x)\\
u & \longmapsto & \gamma_{x,u}
\end{array}
$$
is one-to-one.
\end{proposition}

Given $x\in M$, the end-point-mapping from $x$, from the control viewpoint, takes the following form
$$
\begin{array}{rcl}
\textrm{E}_x: \mathcal{U}_x & \longrightarrow & M\\
u & \longmapsto & \gamma_{x,u}(1)
\end{array}
$$
This mapping is smooth. The derivative of the end-point mapping from $x$ at $u \in \mathcal{U}_x$, that we shall denote by $d\textrm{E}_x (u)$,  is given by
\begin{equation*}
d\textrm{E}_{x}(u) (v) = d\Phi^u (1,x) \int_0^1 \left( d\Phi^u (t,\bar{x}) \right)^{-1}
\biggl( \sum_{i=1}^m v_i(t) f_i(\gamma_{x,u}(t)) \biggr) dt \quad \forall v \in L^2([0,1],\R^m),
\end{equation*}
where $\Phi^u (t,x)$ denotes the flow of the time-dependent  vector field $X^u$ defined by
$$
X^u(t,x)  := \sum_{i=1}^m u_i(t) f_i (x) \qquad \mbox{for a.e. } t \in [0,1], \quad \forall x \in \mathcal{V},
$$
(note that the flow is well-defined in a neighborhood of $x$). We say that an admissible control $u$ is \textit{singular} with respect to $x$ if
$d\textrm{E}_x$ is singular at $u$. Observe that this is equivalent to say that
its associated horizontal path is singular (see the definition of singular path given in Section \ref{sect:prelimresults}). It is important to notice that the singularity of a given horizontal path does not depend on the metric but only on the distribution.

\subsection{Characterization of singular horizontal paths}

Denote by $\omega$ the canonical symplectic form on $T^*M$ and by
$\Delta^\perp$ the annihilator of $\Delta$ in $T^*M$ minus its
zero section. Define $\overline{\omega}$ as the restriction of
$\omega$ to $\Delta^\perp$. An absolutely continuous curve
$\psi:[0,1]\rightarrow \Delta^\perp$ such that
$$
\dot{\psi}(t) \in \ker \overline{\omega}(\psi(t)) \qquad \mbox{for
a.e. } t\in [0,1]
$$
is called an \textit{abnormal extremal} of $\Delta$.

\begin{proposition}
\label{propsing1}
A horizontal path $\gamma :[0,1]\rightarrow M$ is singular if and only if it is the projection of an abnormal extremal
$\psi$ of $\Delta$. The curve $\psi$ is said to be an \textit{abnormal extremal lift} of $\gamma$.
\end{proposition}

If the distribution is parametrized by a family of $m$ smooth vector fields $f_1,\ldots ,f_m$ on some open set $\mathcal{V} \subset M$, and if in addition the cotangent bundle $T^*M$ is trivializable over $\mathcal{V}$, then the singular controls, or equivalently the singular horizontal paths which are contained in $\mathcal{V}$, can be characterized as follows. Define the pseudo-Hamiltonian $H_0 : \mathcal{V} \times (\R^n)^* \times (\R^m) \longmapsto \R$ by
$$
H_0 (x,p,u) = \sum_{i=1}^m u_i p(f_i(x)).
$$

\begin{proposition}
\label{propsing2}
Let $x\in \mathcal{V}$ and $u$ be an admissible control with respect to $x$ and $\mathcal{V}$. Then, the control $u$ is singular (with respect to $x$) if and only if there is an arc $p:[0,1] \longrightarrow (\R^n)^* \setminus \{0\}$ in $W^{1,2}$ such that the pair $(x=\gamma_{x,u},p)$ satisfies
\begin{equation}
\label{abnormal1}
\left\{ \begin{array}{rcl}
\dot{x}(t) & = & \frac{\partial H_0}{\partial p} (x(t),p(t),u(t)) = \sum_{i=1}^m u_i(t)  f_i(x(t)) \\
\dot{p}(t) & = & -\frac{\partial H_0}{\partial x} (x(t),p(t),u(t)) = - \sum_{i=1}^m u_i(t) p(t) \cdot df_i(x(t))
\end{array}
\right.
\end{equation}
for a.e. $t\in [0,1]$  and
\begin{equation}
\label{abnormal2}
p(t)\cdot f_i(x(t)) =0 \qquad \forall t \in [0,1], \quad\forall i=1,\ldots, m.
\end{equation}
\end{proposition}

A control or a horizontal path which is singular is sometimes called \textit{abnormal}. If it is not singular, we call it \textit{nonsingular} or \textit{regular}.

\subsection{Sub-Riemannian minimizing geodesics}

As we said in Section \ref{sect:prelimresults}, since the metric space $(M,d_{SR})$ is assumed to be
complete, for every pair $x,y\in M$ there is a horizontal
path $\gamma $ joining $x$ to $y$ such that
$$
d_{SR}(x,y) = \mbox{length}_{g}(\gamma ).
$$
If $\gamma$ is parametrized by arc-length, then using Cauchy-Schwarz inequality it is easy to show that $\gamma $ minimizes the quantity
$$
\int_0^1 g_{\gamma(t)} (\dot{\gamma}(t),\dot{\gamma}(t)) dt =: \mbox{energy}_g (\gamma),
$$
over all horizontal paths joining $x$ to $y$. This infimum,
denoted by  $e_{SR}(x,y)$, is called  the \textit{sub-Riemannian
energy} between $x$ and $y$. Since $M$ is assumed to be complete,
the infimum is always attained, and the horizontal paths which
minimize the sub-Riemannian energy are those which minimize the
sub-Riemannian distance and which are parametrized by arc-length.
In particular, one has
$$
e_{SR}(x,y) = d_{SR}^2(x,y) \qquad \forall x,y \in M.
$$
Assume from now that $\gamma$ is a given horizontal path
minimizing the energy between $x$ and $y$. Such a path is called a
\textit{sub-Riemannian minimizing geodesic}. Since $\gamma$
minimizes also the distance, it has no self intersection. Hence we
can parametrize the distribution along $\gamma$: there is an open
neighborhood $\mathcal{V}$ of $\gamma([0,1])$ in $M$ and an
orthonormal family (with respect to the metric $g$) of $m$ smooth
vector fields $f_1,\ldots ,f_m$ such that
$$
\Delta (z) = \SPAN \left\{f_1 (z),\ldots, f_m (z) \right\} \qquad \forall z \in \mathcal{V}.
$$
Moreover, since $\gamma$ belongs to $W^{1,2}([0,1],M)$, there
exists a control $u^{\gamma} \in L^2([0,1],\R^m)$ (in fact,
$|u^{\gamma}(t)|^2$ is constant), which is admissible with respect
to $x$ and $\mathcal{V}$, such that
$$
\dot{\gamma}(t) = \sum_{i=1}^m u_i^{\gamma} (t) f_i(\gamma (t)) dt \qquad  \mbox{for a.e. } t\in [0,1].
$$
By the discussion above, we know that $u^{\gamma}$ minimizes the quantity
$$
\int_0^1 g_{\gamma_{x,u}(t)} \biggl(\sum_{i=1}^m u_i(t) f_i(\gamma_{x,u}(t) ),  \sum_{i=1}^m u_i(t) f_i(\gamma_{x,u}(t) ) \biggr) dt = \int_0^1 \sum_{i=1}^m u_i(t)^2 dt =: C(u),
$$
among all controls $u\in L^2([0,1],\R^m)$ which are admissible with respect to $x$ and $\mathcal{V}$, and which satisfy the constraint
$$
\textrm{E}_x (u) =y.
$$
By the Lagrange Multiplier Theorem, there is $\lambda \in (\R^n)^*$ and $\lambda_0 \in \{0,1\}$ such that
\begin{equation}
\label{LMT}
\lambda \cdot d\textrm{E}_x(u^{\gamma}) - \lambda_0 dC(u^{\gamma})=0.
\end{equation}
Two cases may appear, either $\lambda_0=0$ or $\lambda_0=1$. By restricting $\mathcal{V}$ if necessary, we can assume  that the cotangent bundle $T^*M$ is trivializable with coordinates $(x,p) \in \R^n \times (\R^n)^*$ over  $\mathcal{V}$.\\

\noindent First case: $\lambda_0=0$.  The linear operator $d\textrm{E}_x(u^{\gamma}):L^2([0,1],\R^m) \rightarrow T_yM$ cannot be onto,  which means that the control $u$ is necessarily singular. Hence there is an arc $p:[0,1] \longrightarrow (\R^n)^* \setminus \{0\}$ in $W^{1,2}$                  satisfying (\ref{abnormal1}) and (\ref{abnormal2}). In other terms, $\gamma = \gamma_{x,u^{\gamma}}$ admits an abnormal extremal lift in $T^*M$. We also says that $\gamma$ is an \textit{abnormal minimizing geodesic}.\\

\noindent Second case:  $\lambda_0=1$. In local coordinates, the Hamiltonian $H$ (defined in \eqref{monHam}) takes the following form:
\begin{equation}
\label{normal0}
H(x,p) = \frac{1}{2} \sum_{i=1}^m \bigl(  p \cdot f_i(x) \bigr)^2 = \max_{u\in \R^m} \biggl\{ \sum_{i=1}^m u_i p \cdot f_i(x) - \frac{1}{2} \sum_{i=1}^m u_i^2 \biggr\}
\end{equation}
for all $(x,p) \in \mathcal{V} \times (\R^n)^*$. Then the
following result holds:
\begin{proposition}
\label{normal}
Equality (\ref{LMT}) with $\lambda_0=1$ yields the existence of an arc $p:[0,1] \longrightarrow (\R^n)^*$ in $W^{1,2}$, with $p(1) = \frac{\lambda}{2}$,
such that the pair $(\gamma=\gamma_{x,u^{\gamma}},p)$ satisfies
\begin{equation}
\label{normal1}
\left\{ \begin{array}{rcl}
\dot{\gamma}(t) & = & \frac{\partial H}{\partial p} (\gamma(t),p(t)) = \sum_{i=1}^m \left[p(t) \cdot f_i(\gamma(t)) \right]  f_i(\gamma(t)) \\
\dot{p}(t) & = & -\frac{\partial H}{\partial x} (\gamma(t),p(t)) = - \sum_{i=1}^m \left[p(t) \cdot f_i(\gamma(t))\right] p(t) \cdot df_i(\gamma(t))
\end{array}
\right.
\end{equation}
for a.e. $t \in [0,1]$ and
\begin{equation}
\label{normal2}
u^{\gamma}_i(t) = p(t) \cdot f_i(\gamma(t))  \qquad \mbox{for a.e. } t\in [0,1], \quad\forall i=1,\ldots, m.
\end{equation}
In particular, the path $\gamma$ is smooth on $[0,1]$. The curve $\gamma$ and the control $u^{\gamma}$ are called normal.
\end{proposition}

The curve $\psi :[0,1] \rightarrow T^*M$ given by $\psi(t) =(\gamma(t),p(t))$ for every $t\in [0,1]$ is a normal extremal
whose the projection is $\gamma$ and which satisfies $\psi(1)= (y,\frac{\lambda}{2})$. We say that $\psi$ is a \textit{normal extremal lift} of $\gamma$. We also say that $\gamma$ is a \textit{normal minimizing geodesic}.\\

To summarize, the minimizing geodesic (or equivalently the minimizing control $u^{\gamma}$) is either abnormal or normal.  Note that  it could be both normal and abnormal. For decades the prevailing wisdom was that every sub-Riemannian minimizing geodesic is normal, meaning that it admits a normal extremal lift. In 1991, Montgomery found the first counterexample to this assertion (see \cite{montgomery94,montgomery02}).

\subsection{The sub-Riemannian exponential mapping}

Let $x\in M$ be fixed. The \textit{sub-Riemannian exponential mapping} from $x$ is defined by
$$
\begin{array}{rcl}
\exp_{x} : T^*_xM & \longrightarrow & M\\
p & \longmapsto & \pi(\psi(1)),
\end{array}
$$
where $\psi $ is the normal extremal so that $\psi(0)=(x,p)$ in local coordinates.
Note that $H(\psi(t))$ is constant along a normal extremal $\psi $, hence we have
$$
\mbox{energy}_g (\pi(\psi)) =  \left( \mbox{length}_g (\pi(\psi)) \right)^2 =2 H( \psi(0)).
$$
The exponential mapping is not necessarily onto. However, since
$(M,d_{SR})$ is complete, the image of the exponential mapping, $
\exp_{x}  (T^*_xM)$ can be shown to contain an open dense subset
of $M$. This result, which was  obtained recently by Agrachev (see
\cite{agrachevNEW}), is a consequence of the following fact (which
appeared in \cite{rt05}, see also \cite{al07}), which is also
crucial in the proofs of Theorems \ref{mainthm}, \ref{thmrectif}.

\begin{proposition}
\label{propexp1} Let $y\in M$, and assume that there is a function
$\phi: M \rightarrow \R$ differentiable at $y$ such that
$$
\phi(y) = d_{SR}^2(x,y) \quad \mbox{and} \quad  d_{SR}^2(x,z) \geq \phi (z) \quad \forall z \in M.
$$
Then there exists a unique minimizing geodesic between $x$ and
$y$, which is the projection of the normal extremal $\psi:[0,1]
\rightarrow T^*M$ satisfying $\psi(1)=(y,\frac{1}{2}d\phi(y))$. In
particular $x=\exp_{y}(-\frac{1}{2}d\phi(y))$.
\end{proposition}

\subsection{The horizontal eikonal equation}

As in the Riemannian case, the sub-Riemannian distance function
from a given point satisfies a Hamilton-Jacobi equation. This fact
is important for the proof of Theorem \ref{thmrectif}. Let us
first recall the definition of viscosity solution:

\begin{definition}{\rm
Let $F:T^*M \times \R \rightarrow \R$ be a given continuous
function, and let $U$ an open subset of $M$. A continuous function
$u:U\rightarrow \R$ is said to be a \textit{viscosity subsolution}
on $U$ of the Hamilton-Jacobi equation
\begin{equation}\label{HJ}
F(x,du(x),u(x))=0
\end{equation}
if and only if, for every $C^{1}$ function $\phi:U\rightarrow \R$
satisfying $\phi \geq u$ we have
$$
\forall x \in U, \qquad \phi(x) =u(x)\quad \Longrightarrow \quad
F(x,d\phi(x),u(x))\leq 0.
$$
Similarly, a continuous function $u:U \rightarrow \R$ is said to
be a  \textit{viscosity supersolution} of (\ref{HJ}) on $U$ if and
only if, for every $C^{1}$ function $\psi:U\rightarrow \R$
satisfying $\psi \leq u$ we have,
$$
\forall x \in U, \qquad \psi(x) =u(x)\quad \Longrightarrow \quad
F(x,d\psi(x),u(x))\geq 0.
$$
A continuous function $u:U \rightarrow \R$ is called a
\textit{viscosity solution} of (\ref{HJ}) on $U$ if it is both a
viscosity subsolution and a viscosity supersolution of (\ref{HJ})
on $U$.}
\end{definition}

\begin{proposition}
\label{PROP1} For every $x\in M$ the function
$f(\cdot)=d_{SR}(x,\cdot)$ is a viscosity solution of the
Hamilton-Jacobi equation
 \begin{equation}\label{HJeik}
H  (y,d f(y) )=\frac{1}{2} \qquad \forall y \in M \setminus \{x\}.
\end{equation}
\end{proposition}

\subsection{Compactness of minimizing geodesics}

The compactness of minimizing curves is crucial to prove
regularity properties of the sub-Riemannian distance. Let us
denote by  $W_{\Delta}^{1,2}([0,1],M)$ the set of horizontal paths
$\gamma :[0,1] \rightarrow M$ endowed with the $W^{1,2}$-topology.
For every $\gamma \in W_{\Delta}^{1,2}([0,1],M)$, the energy of
$\gamma$ with respect to $g$ is well-defined. The classical
compactness result taken from Agrachev \cite{agrachev98} reads as
follows:

\begin{proposition}
\label{propcompact}
For every compact $K\subset M$, the set
$$
\mathcal{K} := \left\{ \gamma \in W_{\Delta}^{1,2} ([0,1],M) \mid \exists\, x,y \in K \mbox{ with } e_{SR}(x,y) = {\rm energy}_g (\gamma) \right\}
$$
is a compact subset of $W^{1,2} ([0,1],M)$.
\end{proposition}

\subsection{Local semiconcavity of the sub-Riemannian distance}

As we said in Section \ref{sect:prelimresults}, the sub-Riemannian
distance can be shown to be locally H\"older continuous on
$M\times M$, but in general it has no reason to be more regular.
Within the next sections, we are going to show that, under
appropriate  assumptions on the sub-Riemannian structure, $d_{SR}$
enjoyes more regularity properties, such as local semiconcavity or
locally Lipschitz regularity.

Recall that $D$ denotes the diagonal of $M \times M$, that is the
set of all pairs of the form  $(x,x)$ with  $x\in M$. Thanks to
Proposition \ref{propcompact}, the following result holds:

\begin{theorem}
\label{THMsc} Let $\Omega$ be an open subset of $M \times M$ such
that, for every pair $(x,y) \in \Omega$ with $x\neq y$, any
minimizing geodesic between $x$ and $y$ is nonsingular. Then the
distance function $d_{SR}$ (or equivalently $d_{SR}^2$) is locally
semiconcave on $\Omega \setminus D$.
\end{theorem}

Since Theorem \ref{THMsc} plays a crucial role in the present
paper and does not appear in this general form in \cite{cr08}, we
prefer to give a sketch of its proof. We refer the  reader to
\cite{cr08,riffordbook} for more details.

\begin{proof}
Let us fix $(x,y)\in \Omega \setminus D$ and show that $d_{SR}$ is
semiconcave in a neighborhood of $(x,y)$ in $M\times M \setminus
D$. Let $\mathcal{U}_x$ and $\mathcal{U}_y$ be two compact
neighborhoods of $x$ and $y$ such that $\mathcal{U}_x \times
\mathcal{U}_y \subset \Omega \setminus D$.  Denote by
$\mathcal{K}$ the set of minimizing horizontal paths $\gamma$ in
$W_{\Delta}^{1,2}([0,1], \R^m)$ such that $\gamma(0) \in
\mathcal{U}_x$ and $ \gamma(1)\in \mathcal{U}_y$. Thanks to
Proposition \ref{propcompact}, $\mathcal{K}$ is a compact subset
of $W^{1,2}([0,1],M)$. Let $(x',y') \in \mathcal{U}_x \times
\mathcal{U}_y$ be fixed. Since $(M,d_{SR})$ is assumed to be
complete, there exists a sub-Riemannian minimizing geodesic
$\gamma_{x',y'}$ between $x'$ and $y'$. Moreover by assumption it
is nonsingular. As before we can parametrize $\Delta$ by a family
of smooth orthonormal vector fields along $\gamma_{x',y'}$, and we
denote by $u^{x',y'}$ the control in $L^2([0,1],\R^m)$
corresponding to $\gamma_{x',y'}$. Since $u^{x',y'}$ is
nonsingular,  there are $n$ linearly independent controls
$v^{x',y'}_1,\ldots v^{x',y'}_n$ in  $L^2([0,1],\R^m)$ such that
the linear operator
$$
\begin{array}{rcl}
\mathcal{E}^{x',y'} :  \R^n & \longrightarrow & \R^n \\
\alpha & \longmapsto & \sum_{i=1}^m \alpha_i d E_{x'} \left(u^{x',y'}\right) \left( v_i^{x',y'}\right)
\end{array}
$$
is invertible.  Set
$$
\begin{array}{rcl}
\mathcal{F}^{x',y'} :  \R^n \times \R^n & \longrightarrow & \R^n \times \R^n\\
(z,\alpha) & \longmapsto & \left(z, E_z \left( u^{x',y'} + \sum_{i=1}^m \alpha_i v_i^{x',y'}\right) \right)
\end{array}
$$
This mapping is well-defined and smooth in a neighborhood of $(x',0)$, satisfies
$$
\mathcal{F}^{x',y'} (x',0) =(x',y'),
$$
and its differential at $(x',0)$ is invertible. Hence, by the
Inverse Function Theorem, there are an open ball
$\mathcal{B}^{x',y'}$ centered at $(x',y')$ in $\R^n \times \R^n$
and a function $\mathcal{G}^{x',y'} :\mathcal{B}^{x',y'}
\rightarrow \R^n \times \R^n$ such that
$$
\mathcal{F}^{x',y'} \circ \mathcal{G}^{x',y'} (z,w) = (z,w) \qquad \forall (z,w) \in \mathcal{B}^{x',y'}.
$$
Denote by $\left(\alpha^{x',y'}\right)^{-1}$ the second component
of $\mathcal{G}^{x',y'}$. From the definition of the
sub-Riemannian energy between two points we infer that for any
$(z,w) \in \mathcal{B}^{x',y'}$ we have
$$
e_{SR} (z,w) \leq \left\| u^{x',y'} + \sum_{i=1}^m \left( \left( \alpha^{x',y'} \right)^{-1} (z,w) \right)_i  v_i^{x',y'}   \right\|_{L^2}^2.
$$
Set
$$
\phi^{x',y'} (z,w) := \left\| u^{x',y'} + \sum_{i=1}^m \left( \left( \alpha^{x',y'}\right)^{-1} (z,w) \right)_i \right\|_{L^2} \qquad \forall (z,w) \in \mathcal{B}^{x',y'}.
$$
We conclude that, for every $(x',y') \in \mathcal{U}_x \times
\mathcal{U}_y$, there is a smooth function $\phi^{x',y'}$ such
that $d_{SR} (z,w)\leq \phi^{x',y'} (z,w)$ for any $(z,w)$ in
$\mathcal{B}^{x',y'}$. By compactness of $\mathcal{K}$ and thanks
to a quantitative version of the Inverse Function Theorem, the
$C^{1,1}$ norms of the functions $\phi^{x',y'}$ are uniformly
bounded and the radii of the balls $\mathcal{B}^{x',y'}$ are
uniformly bounded from below by a positive constant for $ x',y' $
in $\mathcal{U}_x \times \mathcal{U}_y$. Then the result follows
from Lemma \ref{lem0}.
\end{proof}

\subsection{Sub-Riemannian cut locus}\label{para:cut}

For every $x\in M$ the \textit{singular set} of $d_{SR}(x,\cdot)$,
denoted by $\Sigma \left(d_{SR}(x,\cdot)\right)$,  is defined as
the set of points $y \neq x \in M$ where $d_{SR}(x,\cdot)$ (or
equivalently $d_{SR}^2$)  is not continuously differentiable. The
\textit{cut-locus} of $x$ is defined as
$$
\mbox{Cut}_{SR}(x) := \overline{\Sigma \left( d_{SR}(x,\cdot) \right)}
$$
and the \textit{global cut-locus} of $M$ as
$$
\mbox{Cut}_{SR}(M) := \left\{ (x,y) \in M \ \vert \ y \in \mbox{Cut}_{SR} (x) \right\}.
$$
In contrast with the Riemannian case, the sub-Riemannian global
cut-locus of $M$ always contains the diagonal (see
\cite{agrachev98}). A covector $p\in T_x^*M$ is said to be
\textit{conjugate} with respect to $x\in M$ if the mapping
$\exp_x$ is singular at $p$, that is if $d \exp_{x} (p)$ is
singular. For every $x\in M$ we denote by $\mbox{Conj}_{min}(x)$
the set of points $y \in M \setminus \{x\}$ for which there is
$p\in T_x^*M$ which is conjugate with respect to $x$, and such
that
$$
\exp_{x} (p) = y \qquad \mbox{and} \qquad e_{SR}(x,y) = 2 H(x,p).
$$
The following result holds (see \cite{rt07,riffordbook}):

\begin{proposition}
\label{cut1}
Let $\Omega$ be an open subset of $M\times M$. Assume that $\Omega$ is totally geodesically convex  and that the sub-Riemannian distance is locally semiconcave on $\Omega \setminus D$. Then, for every $x\in M$, we have
$$
\bigl( \{x\} \times \mbox{Cut}_{SR}(x) \bigr) \cap \Omega= \bigl( \{x\} \times \left( \Sigma \left( d_{SR}(x,\cdot) \right) \cup \mbox{Conj}_{min}(x) \cup \{x\}\right)  \bigr) \cap \Omega.
$$
Moreover, the set $\left( \{x\} \times \mbox{Cut}_{SR}(x)\right)  \cap \Omega $ has Hausdorff dimension $\leq n-1$, and the function $d_{SR}$ is of class $C^{\infty}$ on the open set $\Omega \setminus \mbox{Cut}_{SR}(M)$.
\end{proposition}

An important property of the Riemannian distance function is that it fails to be semiconvex at the cut locus (see \cite[Proposition 2.5]{cems01}). This property plays a key role in the proof of the differentiability of the transport map. We do not know if that property holds in the sub-Riemannian case:\\

\textbf{Open problem.} Assume that $d_{SR}$ is locally semiconcave
on $M\times M \setminus D$. Let $x,y \in M$, and assume that there
exists a function $\phi:M \rightarrow \R$ twice differentiable at
$y$ such that
$$
\phi(y) = d_{SR}^2(x,y) \quad \mbox{and} \quad d_{SR}^2 (x,z) \geq
\phi (z) \quad \forall z\in M.
$$
Is it true that $y \notin \mbox{Cut}_{SR}(x)$?

\subsection{Locally lipschitz regularity of the sub-Riemannian distance}
\label{subsect:locLip}

Since any locally semiconcave function is locally Lipschitz,
Theorem \ref{THMsc} above gives a sufficient condition that
insures the Lipschitz regularity of $d_{SR}^2$ out of the
diagonal. In \cite{al07} Agrachev and Lee show that, under some
stronger assumption, one can prove global Lipschitz regularity. A
horizontal path $\gamma :[0,1] \rightarrow M$ will be called a
\textit{Goh path} if it admits an abnormal lift
$\psi:[0,1]\rightarrow \Delta^\perp$ which annihilates
$[\Delta,\Delta]$, that is, for every $t\in [0,1]$ and every local
parametrization of $\Delta$ by smooth vector fields $f_1,\ldots
,f_m$ in a neighborhood of $\gamma(t)$, we have
$$
\psi (t) \cdot \bigl( [f_i,f_j](\gamma(t)) \bigr)=0 \qquad \forall
i,j=1,\ldots, m.
$$
Note that if the path $\gamma$ is constant on $[0,1]$, it is a Goh path if and only if there is a differential form $p\in T_{\gamma(0)}^*M$ satisfying
$$
p \cdot f_i(\gamma(0))= p\cdot [f_i,f_j](\gamma(0)) =0 \qquad \forall i,j=1,\ldots ,m,
$$
where $f_1,\ldots ,f_m$ is as above a parametrization of $\Delta$ in a neighborhood of $\gamma(0)$.  Agrachev and Lee proved the following result (see \cite[Theorem 5.5]{al07}):

\begin{theorem}
\label{THMlip} Let $\Omega$ be an open subset of $M \times M$ such
that any sub-Riemannian minimizing geodesic joining two points of
$\Omega$ is not a Goh path. Then the function $d_{SR}^2$ is
locally Lipschitz on $\Omega  \times \Omega$.
\end{theorem}

\section{Proofs of the results}
\label{sect:proof}

\subsection{Proof of Theorem \ref{mainthm}}

Let us first prove (i). We easily see that $\Mov^\phi$ coincides with the set
$$
\{x\in M\mid
\phi(x) + \phi^c(x) < 0\}.
$$
Thus, since both $\phi$ and $\phi^c$ are continuous,  $\Mov^\phi$ is open. Let us now prove that $\phi$ is locally semiconcave (resp. locally Lipschitz) in an open neighborhood of $\Mov^\phi\cap\supp(\mu)$. Let $x \in \Mov^\phi\cap\supp(\mu)$ be fixed. Since $x \not\in \p^c\phi(x)$, there is $r>0$ such that $d_{SR}(x,y) > r$ for any $y \in
\p^c\phi(x)$. In addition, since the set $\p^c\phi$ is closed in $M\times M$ and $\supp(\mu\times\nu)\subset\O$, there
exists a neighborhood $\mathcal{V}_x$ of $x$ which is included in $\Mov^\phi\cap\pi_1(\O)$ and such that
$$
d_{SR}(x,w) > r \qquad \forall z \in \mathcal{V}_x, \quad \forall
w \in \p^c\phi(z).
$$
Let $\phi_{x,r}: M \rightarrow \R$ be the function defined by
$$
\phi_{x,r}(z):=\inf  \left\{ d_{SR}^2(z,y) - \phi^c(y) \mid y
\in \supp(\nu), \, d_{SR}(z,y) > r \right\}.
$$
We recall that $\supp(\mu\times\nu)\subset\O$ and that $d_{SR}^2$
is locally semiconcave (resp. locally Lipschitz) in $\O \setminus
D$. Thus, up to considering a smaller $\mathcal{V}_x$, we easily get that
the function $\phi_{x,r}$ is locally semiconcave (resp. locally Lipschitz) in
$\mathcal{V}_x$. Since $\phi=\phi_{x,r}$ in $\mathcal{V}_x$, (i)
is proved.

To prove (ii), we observe that it suffices to show the result for
$x$ belonging to an open set $\mathcal{V} \subset M$ on which the
horizontal distribution $\Delta(x)$ is parametrized by a
orthonormal family a smooth vector fields
$\{f_{1},\ldots,f_{m}\}$. Moreover, up to working in charts, we
can assume that $\mathcal{V}$ is a subset of $\R^n$.

First of all we remark that, since all functions $z \mapsto
d_{SR}^2(z,y) - \phi^c(y)$ are locally uniformly Lipschitz with
respect to the sub-Riemannian distance when $y$ varies in a
compact set, also $\phi$ is locally Lipschitz with respect to
$d_{SR}$. Up to a change of coordinates in $\R^n$, we can assume
that the vector fields $f_i$ are of the form
$$
f_i=\frac{\p}{\p x_i} +\sum_{j=m+1}^n a_{ij}(x)\frac{\p}{\p x_j}
\qquad \forall i=1,\ldots,m,
$$
with $a_{ij} \in C^{\infty}(\R^n)$. Therefore, thanks to
\cite[Theorem 3.2]{monsc01}, for a.e. $x \in \mathcal{V}$, $\phi$
is differentiable with respect to all vector fields $f_i$ for a.e.
$x \in \mathcal{V}$, and
\begin{equation}
\label{horizzdiff} \phi(y)-\phi(x)-\sum_{i=1}^m
f_i\phi(x)(y_i-x_i)=o\bigl(d_{SR}(x,y)\bigr) \qquad \forall y \in \mathcal{V}.
\end{equation}
Recalling that $\mu$ is absolutely continuous, we get that
\eqref{horizzdiff} holds at $\mu$-a.e. $x \in \mathcal{V}$.
Thus it suffices to prove that $\p^c\phi(x)=\{x\}$ for all such
points.

Let us fix such an $x$. We claim that
\begin{equation}
\label{claim} f_i \phi (x) = 0 \qquad \forall i=1,\cdots m.
\end{equation}
Indeed, fix $i\in \{1,\cdots,m\}$ and denote by $\g_i^x(t)
:(-\e,\e)\rightarrow M$ the integral curve of the
vector field $f_i$ starting from $x$, i.e.
$$
\left\{
\begin{array}{l}
\dot \g_i^x(t)=f_i(\g_i^x(t)) \qquad \forall t\in (-\epsilon,\epsilon) \\
\g_i^x(0)=x.
\end{array}
\right.
$$
By the assumption on $x$, there is a real number $\ell_i$ such
that
$$
\lim_{t \rightarrow 0} \frac{\phi(\g_i^x(t)) - \phi (x)}{t}
=\ell_i.
$$
By construction, the curve $\g_i^x$ is horizontal with respect to
$\Delta$. Thus, since $g(\dot\g_i^x(t),\dot\g_i^x(t))=1$ for any
$t$, we have
$$
d_{SR}(x,\g_i^x(t)) \leq |t| \qquad \forall t\in
(-\e,\e).
$$
This gives
$$
\phi(\g_i^x(t)) \leq \phi(x) + d_{SR}^2(\g_i^x(t),x) \leq \phi(x) + t^2,
$$
which implies that $\ell_i=0$ and proves the claim.

Assume now by contradiction that there exists a point $y \in
\p^c\phi(x)\setminus\{x\}$, with $(x,y)\in \O$. Then the function
$$
z \mapsto \phi(z)-d_{SR}^2(z,y) \leq \phi^c(x)
$$
attains a maximum at $x$. Let $\g_{x,y}:[0,1]\to M$ denotes a
minimizing geodesic from $x$ to $y$. Then
$$
\phi(\g_{x,y}(t))-d_{SR}^2(\g_{x,y}(t),y) \leq
\phi(x)-d_{SR}^2(x,y) \qquad \forall t \in [0,1],
$$
or equivalently
$$
\phi(\g_{x,y}(t))-\phi(x)\leq
d_{SR}^2(\g_{x,y}(t),y)-d_{SR}^2(x,y) \qquad \forall t \in [0,1].
$$
Observe now that, by \eqref{horizzdiff} together with
\eqref{claim}, we have
$$
\phi(\g_{x,y}(t))-\phi(x)=o\bigl(d_{SR}(\g_{x,y}(t),x)
\bigr)=o\bigl(td_{SR}(x,y)\bigr).
$$
On the other hand, $d_{SR}^2(\g_{x,y}(t),y)=(1-t)^2d_{SR}^2(x,y)$.
Combining all together, for all $t \in [0,1]$ we have
\begin{align*}
o\bigl(td_{SR}(x,y)\bigr)=\phi(\g_{x,y}(t))-\phi(x)
&\leq d_{SR}^2(\g_{x,y}(t),y)-d_{SR}^2(x,y)\\
&=-2td_{SR}^2(x,y)+o\bigl(td_{SR}(x,y)\bigr),
\end{align*}
that is
$$
2t d_{SR}^2(x,y) \leq o\bigl(td_{SR}(x,y)\bigr) \qquad \forall t
\in [0,1].
$$
As $x \neq y$, this is absurd for $t$ small enough, and
the proof of (ii) is completed.

Since $\supp(\mu\times\nu)\subset\O$, we immediately have that any optimal plan $\gamma$ is concentrated on $\p^c\phi\cap \O$.
Moreover, combining (i) and (ii), we obtain that $\p^c\phi(x))\cap \supp(\nu)$ is a singleton for $\mu$-a.e. $x$.
This easily gives existence and uniqueness of the optimal transport map.

To prove the formula for $T(x)$, we have to show that
$$
\p^c\phi(x)\cap \supp(\nu)=\exp_x \left(- \frac{1}{2}d\phi(x) \right)
$$
for all $x\in \Mov^\phi\cap\supp(\mu)$ where $\phi$ is differentiable. This is a consequence of Proposition
\ref{propexp1} applied to the function $z\mapsto \phi(z)+\phi^c(y)$ at the point $x$.
Moreover, again by Proposition \ref{propexp1}, the geodesic from $x$ to $T(x)$ is unique for $\mu$-a.e. $x\in \Mov^\phi\cap\supp(\mu)$.
Since $T(x)=x$ for $x\in \Stat^\phi\cap\supp(\mu)$, the geodesic is clearly unique also in this case.

\subsection{Proof of Theorem \ref{thmrectif}}
We will prove only (ii), as all the rest follows as in the proof of Theorem \ref{mainthm}.

Let us consider the ``bad'' set defined by
$$
\Bad:=\bigl\{x \in \Stat^\phi\cap\supp(\mu) \mid
\left( \p^c\phi(x) \setminus \{x\} \right) \cap \supp(\nu)\neq\emptyset\bigr\}.
$$
We have to show that $\Bad$ is $\mu$-negligible.
For each $k \in \N$, we consider the sequence of function
constructed as follows:
$$
\phi_k(x):=\inf  \left\{ d_{SR}^2(x,y) - \phi^c(y) \mid y
\in \supp(\nu), \, d_{SR}(x,y) > 1/k \right\}.
$$
Since $\supp(\mu\times\nu) \subset \O$ and $d_{SR}^2$ is locally
semiconcave in $\O\setminus D$, the functions $\phi_k$ are
locally semiconcave in a neighborhood of $\Bad$.

Thus, by Theorem \ref{THMsingset} and the assumptions on $\mu$, there exists a Borel set $G$, with
$\mu(G)=1$, such that all $\phi_k$ are differentiable in $G$.
Since for any $x \in \Bad$ there exists $y \in \p^c\phi(x) \setminus
\{x\}$ such that $d_{SR}(y,x)>1/k$ for some $k$, we deduce that
$$
\bigcup_{k \in \N}\{\phi=\phi_k\}=\Bad.
$$
This gives that, up to set of $\mu$-measure zero, $\Bad$ coincides with $\cup_{k \in \N} A_k$, where
$$
A_k:=\Bad \cap \{\phi=\phi_k\} \cap G.
$$
Hence, to conclude the proof, it suffices to show that
$\mu(A_k)=0$ for all $k\in \N$.

Let $x \in A_k$. Then, if $y \in \p^c\phi(x)$ and $d_{SR}(x,y)>1/k$, the function
\begin{equation}
\label{eqmaxphik}
z \mapsto \phi_k(z)-d_{SR}^2(z,y) \leq \phi^c(x)
\end{equation}
attains a maximum at $x$.
Therefore, if we show that $d\phi_k(x)=0$ for
$\mu$-a.e. $x \in A_k$, equation \eqref{eqmaxphik} together with the semiconcavity of $d_{SR}^2(z,y)$ for $z$ close to $x$
would imply that $d_{SR}^2(\cdot,y)$ is differentiable at $x$, and its differential is equal to $0$.
This would contradict Proposition \ref{PROP1}, concluding the proof.
Therefore we just need to show that $d\phi_k(x)=0$ $\mu$-a.e. in $A_k$.\\

Let $X$ be a smooth section of $\Delta$ such that $g_x(X(x),X(x))=1$ for any $x\in M$. We claim the following:\\
\textbf{Claim 1:} for $\mu$-a.e. $x \in A_k$, $d\phi_k(x)\cdot X(x)\leq 0$.\\
Since we can apply Claim 1 with a countable set of vector fields
$\{X_\ell\}_{\ell \in \N}$ such that $\{X_\ell(x)\}_{\ell \in \N}$
is dense in $\Delta(x)$ for all $x \in \supp(\mu)$, Claim 1
clearly implies that $d\phi_k(x)=0$ $\mu$-a.e. in $A_k$. Let us
prove the claim.

Let $d_g$ denote the Riemannian distance associated to the Riemannian metric $g$, and $\theta(x,t)$ denote the flow of $X$, that is the function $\theta:M\times \R \rightarrow M$ satisfying
$$
\frac{d}{dt} \theta(x,t) = X(\theta(x,t)), \qquad \theta(x,0)=x.
$$
Fix $\e>0$ small, and consider the ``cone'' around the curve $t\mapsto \theta(x,t)$ given by
$$
C_x^\e:=\Bigl\{y\in \O \mid \text{$\exists \, t \in [0,\e]$ such that $d_g\bigl(\theta(x,t),y\bigl) \leq \e t$}\Bigr\}.
$$
Moreover we define
$$
R_\e:=\bigl\{x \in \supp(\mu)\cap A_k \mid A_k \cap C_x^\e=\{x\} \bigr\}.
$$
\textbf{Claim 2:} $R_\e$ is countably $(n-1)$-rectifiable for any $\e>0$.\\
Indeed, since the statement is local, we can assume that we are in $\R^n$,
Moreover, since $X$ is smooth, we can assume that there exists $\bar v \in \R^n$
such that $C_x^\e$ contains the ``euclidean cone''
$$
\bar C_x^{\e/2}:=\Bigl\{y\in \O \mid \text{$\exists \, t \in
[0,\e/2]$ such that $|x+t\bar v-y| \leq c_0\frac{\e}{2}
t$}\Bigr\},
$$
where $c_0>0$. Thus it suffices to prove that
$$
\bar R_{\e/2}:=\bigl\{x \in \supp(\mu)\cap A_k \mid A_k \cap \bar C_x^{\e/2}=\{x\}\bigr\}
$$
is $(n-1)$-rectifiable for any $\e >0$.\\
Assume now that $z, z' \in \bar R_{\e/2}$, with $z\neq z'$.
Then, since $z\not\in \bar C_{z'}^{\e/2}$, we have
$$
|z'+t\bar v-z| >c_0\frac{\e}{2} t \qquad \forall t \in [0,\e/2],
$$
or equivalently
$$
|z-t\bar v-z'| > c_0\frac{\e}{2} t\qquad \forall t \in [0,\e/2].
$$
This implies that
$$
z' \not \in \bar C_{z}^{\e/2,-}:=\Bigl\{y\in \O \mid
\text{$\exists \, t \in [0,\e/2]$ such that $|x-t\bar v-y| \leq
c_0 \frac{\e}{2} t$}\Bigr\}.
$$
Since $z,z' \in \bar R_{\e/2}$ were arbitrary, we have proved that
for all $z\in \bar R_{\e/2}$
$$
\bar R_{\e/2}\cap \bigl(\bar C_z^{\e/2}\cup\bar C_z^{\e/2,-}\bigr)=\{z\}.
$$
By \cite[Theorem 4.1.6]{cs04} $\bar R_\e$ is countably $(n-1)$-rectifiable for any $\e>0$, and this concludes the proof of Claim 2.\\

Let us come back to the proof of Claim 1. Thanks to Claim 2 we just need to show that
$$
x \in \bigl(\supp(\mu)\cap A_k) \setminus \bigl( \cup_j R_{1/j}\bigr) \quad \Longrightarrow \quad d\phi_k(x)\cdot X(x)\leq 0.
$$
Let $x \in \bigl(\supp(\mu)\cap A_k) \setminus \bigl( \cup_j R_{1/j}\bigr)$. Then
$\phi(x)=\phi_k(x)$, and there exists a sequence of points
$\{x_j\}$ such that $x_j\neq x$ and $x_j \in A_k \cap C_x^{1/j}$ for all $j\in \N$. In particular $\phi(x_j)=\phi_k(x_j)$ for all $j \in \N$.
Since $x\in\Stat^\phi$, we have $x\in\p^c\phi(x)$, and so
$$
\phi(z) -\phi(x) \leq d_{SR}^2(z,x) \qquad \forall z \in M.
$$
Let $t_j \in [0,\frac{1}{j}]$ be such that $d_g\bigl(\theta(x,t_j),x_j\bigl) \leq \frac{1}{j} t_j$.
Then, since $d_{SR}^2$ is locally Lipschitz, we get
\begin{align*}
\phi_k(x_j) -\phi_k(x) &=\phi(x_j) -\phi(x)\leq d_{SR}^2(x_j,x)\\
&\leq 2d_{SR}^2\bigl(\theta (x,t_j),x_j\bigl) + 2d_{SR}^2\bigl( \theta(x,t_j),x\bigr)\\
&\leq Cd_g\bigl(\theta (x,t_j),x_j\bigl) + 2d_{SR}^2\bigl(\theta (x,t_j),x\bigr)\\
&\leq \frac{C}{j} t_j + 2d_{SR}^2\bigl(\theta (x,t_j),x\bigr).
\end{align*}
We now observe that, since $X$ is a unitary horizontal vector field, $d_{SR}\bigl(\theta (x,t_j),x\bigr)
\leq t_j$.
Moreover $t_j = d_g(x_j,x)+o\bigl(d_g(x_j,x)\bigr)$ as $j\to \infty$. Therefore, up to subsequences, one easily gets (looking everything in charts)
$$
\lim_{j\rightarrow +\infty} \frac{x_j -x  }{d_g(x_j,x)} = X(x),
$$
which implies
$$
d\phi_k(x)\cdot X(x)\leq 0,
$$
as wanted.

\subsection{Proof of Theorem \ref{thmabscont}}
Let us first prove the uniqueness of the Wasserstein geodesic. A
basic representation theorem (see \cite[Corollary
7.22]{villaniSF}) states that any Wasserstein geodesic
necessarily takes the form $\mu_t=(e_t)_\# \Pi$, where $\Pi$ is a
probability measure on the set $\Gamma$ of minimizing geodesics
$[0,1]\to M$, and $e_t:\Gamma \rightarrow M$ is the evaluation at
time $t$: $e_t(\g):=\g(t)$. Thus uniqueness follows easily from
Theorem \ref{mainthm}.

The proof of the absolute continuity of $\mu_t$ is done as
follows. Fix $t \in (0,1)$, and define the functions
$$
\phi_{1-t}(x):=\inf_{y \in \supp(\nu)}  \left\{ \frac{d_{SR}^2(x,y)}{1-t} -
\phi^c(y) \right\},
$$
$$
\phi_{t}^c(y):=\inf_{x \in \supp(\mu)} \left\{ \frac{d_{SR}^2(x,y)}{t} -
\phi(x)\right\}.
$$
It is not difficult to see that
\begin{equation}
\label{triangleineq} \frac{d_{SR}^2(x,z)}{t} +
\frac{d_{SR}(z,y)^2}{1-t} \geq d_{SR}^2(x,y) \qquad \forall x,y,z
\in M.
\end{equation}
Indeed, for all $\e>0$,
$$
d_{SR}^2(x,y) \leq \bigl(d_{SR}(x,z) + d_{SR}(z,y) \bigr)^2\leq ( 1+\e )d_{SR}^2(x,z) + \Bigl( 1+\frac{1}{\e}\Bigr) d_{SR}^2(z,y).
$$
Choosing $\e>0$ so that $1+\e=1/t$, \eqref{triangleineq} follows.
Since $\phi(x)+\phi^c(y) \leq d_{SR}^2(x,y)$ for all $x
\in \supp(\mu)$ and $y\in \supp(\nu)$, by \eqref{triangleineq} we get
$$
\Bigl[\frac{d_{SR}(z,y)^2}{1-t} - \phi^c(y) \Bigr]+ \Bigl[
\frac{d_{SR}^2(x,z)}{t} - \phi(x)\Bigr] \geq 0 \qquad \forall x
\in \supp(\mu),\,y \in \supp(\nu), \,z \in M.
$$
This implies
\begin{equation}
\label{ineqsemic} \phi_{1-t}(z) + \phi^c_{t}(z) \geq 0 \qquad
\forall z \in M.
\end{equation}
We now remark that \eqref{triangleineq} becomes an equality if and
only if there exists a geodesic $\g:[0,1] \to M$ joining $x$ to
$y$ such that $z=\g(t)$. Hence by the definition of $T_t(x)$ we
get
\begin{equation}
\label{eqtransport} \frac{d_{SR}(x,T_t(x))^2}{t} +
\frac{d_{SR}(T_t(x),T(x))^2}{1-t} = d_{SR}^2(x,T(x)) \qquad
\text{for $\mu$-a.e. $x$}.
\end{equation}
Moreover, since
$$
\phi(x)+\phi^c(T(x))=d_{SR}^2(x,T(x)) \qquad \text{for $\mu$-a.e.
$x$},
$$
we obtain
$$
\phi_{1-t}(T_t(x)) + \phi^c_{t}(T_t(x)) = 0 \qquad \text{for
$\mu$-a.e. $x$},
$$
or equivalently
\begin{equation}
\label{eqsemic} \phi_{1-t}(z) + \phi^c_{t}(z) =0 \qquad \text{for
$\mu_t$-a.e. $z$}.
\end{equation}

Let us now decompose the set $\Mov^\phi\cap \supp(\mu)$ as
$$
A_k:=\{x \in \Mov^\phi\cap \supp(\mu) \mid d_{SR}(x,y)
>1/k\quad\forall y \in \p^c\phi(x)\}.
$$
Since $T_t(x)=x$ on $\Stat^\phi\cap \supp(\mu)$, defining
$\mu_t^k:=\mu_t \lfloor_{T_t(A_k)}$ we have
$$
\mu_t=\bigl(\cup_{k}\mu_t^k \bigr) \cup \mu
\lfloor_{\bigl(\Stat^\phi\cap \supp(\mu)\bigr)} \qquad \forall t
\in [0,1].
$$
Thus it suffices to prove that $\mu_t^k$ is
absolutely continuous for each $k \in \N$.

We consider the functions
$$
\phi_{k,1-t}(x):=\inf\left\{\frac{d_{SR}^2(x,y)}{1-t} - \phi^c(y)
\mid y \in \supp(\nu), \, d_{SR}(x,y) > (1-t)/k \right\}.
$$
$$
\phi_{k,t}^c(y):=\inf\left\{\frac{d_{SR}^2(x,y)}{t} - \phi(x)
\mid y \in \supp(\nu), \, d_{SR}(x,y) > t/k \right\}.
$$
Since $d_{SR}(x,T(x))>1/k$ for $x \in A_k$, they coincide
respectively with $\phi_{1-t}$ and $\phi_{t}^c$ inside $T_t(A_k)$.
Thus, thanks to \eqref{ineqsemic} and \eqref{eqsemic} we have
$$
\phi_{k,1-t}(z) + \phi_{k,t}^c(z) \geq \phi_{1-t}(z) +
\phi_{t}^c(z) \geq 0\qquad \forall z \in M,
$$
with equality $\mu_t$-a.e. on $T_t(A_k)$.

Observe now that, by the compactness of the supports of $\mu$
and $\nu$, and the fact that $\O$ is totally geodesically convex,
$\supp(\mu\times\mu_t)$ and $\supp(\mu_t\times\nu)$ are compact and contained in $\O$.
Thus, since $d_{SR}^2$ is locally semiconcave on $\O \setminus D$, both functions $\phi_{k,1-t}$ and
$\phi_{k,t}^c$ are locally semiconcave in a neighborhood of $T_t(A_k)$.
It follows from
\cite[Theorem A.19]{fatfig} that both differentials
$d\phi_{k,t}(z)$, $d\phi_{k,1-t}^c(z)$ exist and are equal
for $\mu$-a.e. $z\in T_s(A_k)$. Moreover, again by \cite[Theorem A.19]{fatfig}, the map $z\mapsto
d\phi_{k,t}(z)=d\phi_{k,1-t}^c(z)$ is locally Lipschitz on
$T_s(A_k)$. Since for  $x \in A_k$ we have
$$
\phi_{k,t}(\cdot) \leq \frac{d_{SR}(x,\cdot)^2}{t} - \phi(x)\qquad
\text{on }\{z \mid d_{SR}(x,z) >t/k\}
$$
with equality at $T_t(x)$ for $\mu$-a.e. $x \in A_k$, by
Proposition \ref{propexp1} we get
$$
x=\exp_{T_t(x)}(-\textstyle{\frac{1}{2}}d\phi_{k,t}(T_t(x))) \qquad \text{for
$\mu$-a.e. $x \in A_k$}.
$$
Denoting by $\Phi_t:T^*M \to T^*M$ the Euler-Lagrange flow
(i.e. the flow of the Hamiltonian vector field $\overrightarrow{H}$), we see
that the map
$$
F_{t,k}(z):=\exp_{z}(-\textstyle{\frac{1}{2}}d\phi_{k,t}(z))=\Phi_t(z,-\textstyle{\frac{1}{2}}d\phi_{k,t}(z))
$$
is locally Lipschitz on $\supp(\mu_t) \cap T_t(A_k)$. Therefore it
is clear that $\mu_t^k$ cannot have a singular part with respect
to the volume measure, since otherwise the same would be true for
$(F_{t,k})_\# (\mu_t^k)=\mu \lfloor_{A_k}$. This concludes the
proof of the absolute continuity.

\subsection{Proof of Theorem \ref{thmapproxdiff}}
We recall that, by Theorem \ref{mainthm}, the function $\phi$ is
locally semiconcave in a neighborhood of
$\Mov^\phi\cap\supp(\mu)$. Thus, since $\mu$ is absolutely
continuous with respect to the volume measure, by Theorem
\ref{alexandrov} $d\phi(x)$ is differentiable for $\mu$-a.e. $x\in
\Mov^\phi\cap\supp(\mu)$. By Theorem \ref{mainthm}, for $\mu$-a.e.
$x$ there exists a unique minimizing geodesic between $x$ and
$T(x)$. Thanks to our assumptions this implies that
$T(x)=\exp_x(-\frac{1}{2}d\phi(x))$ do not belongs to
$\mbox{Cut}_{SR}(x)$ for $\mu$-a.e. $x \in
\Mov^\phi\cap\supp(\mu)$. Hence Proposition \ref{cut1} implies
that the function
$$
(z,w)\mapsto d_{SR}^2(z,w)
$$
is smooth near $(x,T(x))$. Exactly as in the Riemannian case, this
gives that the map $x \mapsto \exp_x(-\frac{1}{2}d\phi(x))$ is
differentiable for $\mu$-a.e. $x$, and its differential is given
by $Y(x)\bigl(H(x)-\frac{1}{2}\Hess_x^2\phi\bigr)$ (see
\cite[Proposition 4.1]{cems01}). On the other hand, since $T(x)=x$
for $x\in \Stat^\phi\cap\supp(\mu)$, it is clear by Definition
\ref{approxdiff} that $T$ is approximately differentiable
$\mu$-a.e. in $\Stat^\phi\cap\supp(\mu)$, and that its approximate
differential is given by the identity matrix $I$. This proves the
first part of the theorem.

To prove the change of variable formula, we first remark that, since both $\mu$ and $\nu$ are
absolutely continuous, there exists also an optimal transport map
$S$ from $\nu$ to $\mu$, and it is well-known that $S$ is an inverse for $T$ a.e., that is
$$
S \circ T=Id \quad \mu\text{-a.e.}, \qquad T \circ S=Id \quad
\nu\text{-a.e.}
$$
(see for instance \cite[Remark 6.2.11]{agsbook}). This gives in particular that $T$
is a.e. injective. Applying \cite[Lemma 5.5.3]{agsbook} (whose proof is in the Euclidean case, but still works on a manifold) we deduce that
$|\det(\tilde d T(x))|>0$ $\mu$-a.e., and that the Jacobian identity
holds.

\appendix

\section{Locally semiconcave functions}

The aim of this section is to recall some basic facts on semiconcavity. Throughout this section,
$M$ denotes a smooth connected manifold of dimension $n$.

For an introduction to semiconcavity, we refer the reader to
\cite{cs04} and \cite[Appendix A]{fatfig}. A function $u:U \rightarrow \R$, defined on the open set
$U \subset M$, is called \textit{locally semiconcave} on $U$ if for every $x \in U$
there exist a neighborhood $U_x$ of $x$ and a smooth diffeomorphism
$\varphi_x: U_x \rightarrow \varphi_x(U_x) \subset \R^n$ such that
$f\circ \varphi_x^{-1}$ is locally semiconcave on the open subset $\tilde{U}_x =\varphi_x(U_x) \subset \R^n$. We recall that the function
$u: U \rightarrow \R$, defined on the open set $U \subset \R^n$, is
locally semiconcave on $U$ if  for every $\bar{x} \in U$
there exist  $C,\delta >0$ such that
\begin{equation}
    \label{sc1}
\mu u(y)+ (1-\mu) u(x) -u(\mu x +(1-\mu) y) \leq \mu (1-\mu) C |x-y|^2,
\end{equation}
for all $x,y$ in the ball $B_\delta(\bar{x})$ and every $\mu \in [0,1]$. This is equivalent
to say that the function $u$ can be written locally as
$$
u(x) = \left( u(x) - C|x|^2 \right) + C|x|^2  \qquad \forall x \in B_\delta(\bar{x}),
$$
with $u(x) - C|x|^2$ concave. Note that every locally semiconcave
function is locally Lipschitz on its domain, and thus by
Rademacher's Theorem it is differentiable almost everywhere on its
domain (in fact a better result holds, see Theorem
\ref{THMsingset}). The following result will be useful in the
proof of our theorems.

\begin{lemma}
    \label{lem0}
    Let $u:U \rightarrow \R$ be a function defined on an open set
    $U\subset \R^n$. Assume that for every $\bar{x} \in U$ there exist a
    neighborhood $\mathcal{V} \subset U$ of $\bar{x}$ and a positive real number
    $\sigma$ such that, for every  $x\in \mathcal{V}$,
    there is $p_{x} \in \R^n$ such that
    \begin{equation}
        \label{sc2}
    u(y) \leq u(x) + \langle p_{x}, y-x\rangle + \sigma |y-x|^2 \qquad \forall y \in \mathcal{V}.
    \end{equation}
    Then the function $u$ is locally
    semiconcave on $U$.
    \end{lemma}

    \begin{proof}
Let $\bar{x} \in U$ be fixed and $\mathcal{V}$ be the neighborhood given by assumption. Without loss of generality, we can assume that $\mathcal{V}$ is an open ball $\mathcal{B}$. Let $x,y \in \mathcal{B}$ and $\mu \in [0,1]$. The point
     $\hat{x}:= \mu x + (1-\mu) y$ belongs to $\mathcal{B}$. By assumption, there exists $\hat{p} \in \R^n$ such that
     $$
     u(z) \leq u(\hat{x}) + \langle \hat{p}, z-\hat{x} \rangle + \sigma |z-\hat{x}|^2 \qquad \forall z \in \mathcal{B}.
     $$
     Hence we easily get
     \begin{align*}
     \mu u(y) + (1-\mu) u(x) & \leq u(\hat{x}) + \mu \sigma |x-\hat{x}|^2 + (1-\mu) \sigma |y-\hat{x}|^2 \\
                             & \leq u(\hat{x}) +  \left( \mu (1-\mu)^2 \sigma + (1-\mu) \mu^2 \sigma \right) |x-y|^2 \\
                             & \leq u(\hat{x}) + 2 \mu (1-\mu) \sigma |x-y|^2,
        \end{align*}
        and the conclusion follows.
     \end{proof}

Another useful result is the following (see \cite[Corollary
3.3.8]{cs04}):

\begin{proposition}
\label{propSCSC}
Let $u:U \rightarrow \R$ be a function defined on an open set
    $U\subset M$. If both functions $u$ and $-u$ are locally
    semiconcave on $U$, then $u$ is of class $C^{1,1}_{loc}$ on $U$.
\end{proposition}

Fathi generalized the proposition above as follows (see \cite{fathibook} or \cite[Theorem A.19]{fatfig}):

\begin{proposition}
\label{propFATHI}
Let $U$ be an open subset of $M$ and $u_1,u_2: U \rightarrow \R$ be two functions with $u_1$ and $-u_2$ locally semiconcave on $U$.  Assume that $u_1 (x)\leq u_2 (x)$ for any $x\in U$. If we define $\mathcal{E}= \left\{x\in U \ \vert \ u_1(x) =u_2(x)\right\}$, then both $u_1$ and $u_2$ are differentiable at each $x\in \mathcal{E}$ with $du_1(x)=du_2(x)$ at such a point. Moreover, the map $x\mapsto du_1(x)=du_2(x)$ is locally Lipschitz on $\mathcal{E}$.
\end{proposition}


\subsection{Singular sets of semiconcave functions}

Let $u:U\rightarrow \R$ be a function which is locally semiconcave on
the open set $U\subset M$. We recall that, since such a function is locally Lipschitz on $U$, its limiting subdifferential is always nonempty on $U$. We define the \textit{singular set}
of $u$ as the subset of $U$
$$
\Sigma(u) := \left\{ x\in U \mid u \mbox{ is not differentiable at } x \right\}.
$$
From Rademacher's theorem, $\Sigma (u)$ has Lebesgue measure zero. In fact, the following result holds (see \cite{cs04,riffordbook}):

\begin{theorem}
\label{THMsingset} Let $U$ be an open subset of $M$. The singular
set of a locally semiconcave function $u:U \rightarrow \R$ is
countably $(n-1)$-rectifiable, i.e. is contained in a countable
union of locally Lipschitz hypersurfaces of $M$.
\end{theorem}

\subsection{Alexandrov's second differentiability theorem}

As shown by Alexandrov (see \cite{villaniSF}), locally semiconcave functions are two times differentiable almost everywhere.

\begin{theorem}
\label{alexandrov}
Let $U$ be an open subset of $\R^n$ and $u:U \rightarrow \R$ be a function which is locally semiconcave on $U$. Then, for a.e. $x\in U$, $u$ is differentiable at $x$ and there exists a symmetric operator $A(x):\R^n \rightarrow \R^n$  such that the following property is satisfied:
$$
\lim_{t\downarrow 0} \frac{u(x+tv) -u(x) -t du(x)\cdot v - \frac{t^2}{2} \langle A(x)\cdot v,v\rangle }{t^2}=0 \qquad \forall v \in \R^n.
$$
Moreover, $du(x)$ is differentiable a.e. in $U$, and its differential is given by $A(x)$.
\end{theorem}

\section{Proofs of auxiliary results}

\subsection{Proof of Proposition \ref{propdim3}}

The first part of the proposition is just a corollary of Proposition \ref{propcodim1} for $n=3$. Let us prove the second part of the proposition. Let $\gamma:[0,1] \rightarrow M$ be a nontrivial singular horizontal path. Our aim is to show that, for every $t\in [0,1]$, the point $\gamma(t)$ belongs to $\Sigma_{\Delta}$. Fix $\bar{t}\in [0,1]$ and parametrize the distribution by two smooth vector fields $f_1,f_2$ in an open neighborhood $\mathcal{V}$ of $\gamma(\bar{t})$. Let $u\in L^2([0,1],\R^2)$, and let $I$ be an open subinterval of $[0,1]$ containing $\bar{t}$ such that
$$
\dot{\gamma}(t) =  u_1(t) f_1(\gamma(t)) + u_2(t) f_2(\gamma(t)) \qquad \mbox{for a.e. } t\in I.
$$
Note that since $\gamma$ is assumed to be nontrivial, we can
assume that $u $ is not identically zero in any neighborhood of
$\bar{t}$. From Proposition \ref{propsing2} there is an arc
$p:[0,1] \longrightarrow (\R^3)^* \setminus \{0\}$ in $W^{1,2}$
such that
$$
\dot{p}(t)  = - u_1(t) p(t) \cdot df_1(\gamma(t)) -u_2(t) p(t)
\cdot df_2(\gamma(t)) \qquad \mbox{for a.e. } t\in I,
$$
and
$$
p(t)\cdot f_1(\gamma(t)) =  p(t)\cdot f_2(\gamma(t))=  0 \qquad \forall t \in I.
$$
Let us take the derivative of the quantity $p(t)\cdot f_1(\gamma(t))$ (which is absolutely continuous). We have for almost every $t\in I$,
\begin{align*}
0 & = \frac{d}{dt} \left[ p(t)\cdot f_1(\gamma(t))\right] \\
& =  \dot{p}(t)\cdot f_1(\gamma(t)) + p(t)\cdot df_1(\gamma(t))\cdot \dot{\gamma}(t) \\
& =  - \sum_{i=1,2} u_i(t) p(t) \cdot df_i(\gamma(t)) \cdot f_1(\gamma(t))  + \sum_{i=1,2} u_i(t) p(t) \cdot df_1(\gamma(t)) \cdot f_i(\gamma(t)) \\
& =  -u_2(t) p(t)\cdot [f_1,f_2](\gamma(t)).
\end{align*}
In the same way, if we differentiate the quantity  $p(t)\cdot f_2(\gamma(t))$, we obtain
$$
0=\frac{d}{dt} \left[ p(t)\cdot f_2(\gamma(t))\right] = u_1(t) \cdot [f_1,f_2](\gamma(t)).
$$
Therefore, since $u$ is not identically zero in any neighborhood
of $\bar{t}$, thanks to the continuity of the mapping $t\mapsto
p(t)\cdot[f_1,f_2](\gamma(t))$ we deduce that
$$
p(\bar{t})\cdot [f_1,f_2](\gamma(\bar{t}))=0.
$$
But we already know that $ p(t)\cdot f_1(\gamma(\bar{t})) =
p(t)\cdot f_2(\gamma(\bar{t})) =0$, where the two vectors
$f_1(\gamma(\bar{t})), f_2(\gamma(\bar{t}))$ are linearly
independent. Therefore, since $p(\bar{t})\neq 0$,  we conclude
that the Lie bracket  $[f_1,f_2](\gamma(\bar{t}))$ belongs to the
linear subspace spanned by   $f_1(\gamma(\bar{t})),
f_2(\gamma(\bar{t}))$, which means that $\gamma(\bar{t})$ belongs
to $\Sigma_{\Delta}$. Let us now prove that any horizontal path
included in $\Sigma_{\Delta}$ is singular. Let $\gamma$ such a
path be fixed, set $\gamma(0)=x$, and consider a parametrization
of $\Delta$ by two vector fields $f_1,f_2$ in a neighborhood
$\mathcal{V}$ of $x$. Let $\delta>0$ be small enough so that
$\gamma(t) \in \mathcal{V}$ for any $t\in [0,\delta]$, in such a
way that there is $u\in L^2([0,\delta],\R^2)$ satisfying
$$
\dot{\gamma}(t) =  u_1(t) f_1(\gamma(t)) + u_2(t) f_2(\gamma(t)) \qquad \mbox{for a.e. } t\in [0,\delta].
$$
Let $p_0\in (\R^3)^*$ be such that $p_0\cdot f_1(x) =p_0\cdot
f_2(x)=0$, and let $p:[0,\delta] \rightarrow (\R^3)^*$ be the
solution to the Cauchy problem
$$
\dot{p}(t) = - \sum_{i=1,2} u_i(t) p(t)\cdot df_i(\gamma(t)) \qquad \mbox{for a.e. } t\in [0,\delta], \quad p(0)=p_0.
$$
Define two absolutely continuous function $h_1,h_2 :[0,\delta] \rightarrow \R$ by
$$
h_i(t) = p(t)\cdot f_i(\gamma(t)) \qquad \forall t\in [0,\delta], \quad \forall i=1,2.
$$
As above, for every $t\in [0,\delta]$ we have
$$
\dot{h}_1(t) = \frac{d}{dt} \left[p(t)\cdot f_1(\gamma(t))\right] = -u_2(t) p(t)\cdot [f_1,f_2](\gamma(t))
$$
and
$$
\dot{h}_2 (t) = u_1(t) p(t)\cdot [f_1,f_2](\gamma(t)).
$$
But since $\gamma(t)\in \Sigma_{\Delta}$ for every $t$, there are two continuous functions $\lambda_1,\lambda_2:[0,\delta] \rightarrow \R$ such that
$$
[f_1,f_2](\gamma(t)) = \lambda_1(t) f_1(\gamma(t)) + \lambda_2 (t)
f_2(\gamma(t)) \qquad \forall t\in [0,\delta].
$$
This implies that the pair $(h_1,h_2)$ is a solution of the linear differential system
$$
\left\{
\begin{array}{rcl}
\dot{h}_1(t) & = & -u_2(t)\lambda_1(t) h_1(t) -u_2(t)\lambda_2(t) h_2(t) \\
\dot{h}_2(t) & = & u_1(t) \lambda_1(t) h_1(t) +u_1(t) \lambda_2(t) h_2(t).
\end{array}
\right.
$$
Since $h_1(0)=h_2(0)=0$ by construction, we deduce by the
Cauchy-Lipschitz Theorem that $h_1(t)=h_2(t)=0$ for any $t\in
[0,\delta]$. In that way, we have constructed an abnormal lift of
$\gamma$ on the interval $[0,\delta]$. We can in fact repeat this
construction on a new interval of the form $[\delta,2\delta]$
(with initial condition $p(\delta)$)  and finally obtain an
abnormal lift of $\gamma$ on $[0,1]$.  By Proposition
\ref{propsing1}, we conclude that $\gamma$ is singular.

\subsection{Proof of Proposition \ref{propcodim1}}

The fact that $\Sigma_{\Delta}$ is a closed subset of $M$ is obvious. Let us prove that it is countably $(n-1)$-rectifiable.
Since it suffices to prove the result locally, we can assume that we have
$$
\Delta (x) = \SPAN \{f_1(x),\ldots ,f_{n-1}(x)\} \qquad \forall x \in \mathcal{V},
$$
where $\mathcal{V}$ is an open neighborhood of the origin in $\R^n$.  Moreover, doing a change of coordinates if necessary, we can also assume that
$$
f_i = \frac{\partial}{\partial x_i} + \alpha_i (x) \frac{\partial}{\partial x_n} \qquad \forall i=1,\ldots ,n-1,
$$
where each $\alpha_i :\mathcal{V} \longrightarrow \R$ is a
$C^\infty$ function satisfying $\alpha_i (0)=0$. Hence for any
$i,j \in \{1,\ldots n-1\}$ we have
$$
[f_i,f_j] = \left[ \left( \frac{\partial \alpha_j}{\partial x_i} -  \frac{\partial \alpha_i}{\partial x_j}\right) + \left( \frac{\partial \alpha_j}{\partial x_n} \alpha_i - \frac{\partial \alpha_i}{\partial x_n} \alpha_j \right) \right] \frac{\partial}{\partial x_n},
$$
and so
$$
\Sigma_{\Delta} = \left\{ x \in \mathcal{V} \mid
\left( \frac{\partial \alpha_j}{\partial x_i} -  \frac{\partial \alpha_i}{\partial x_j}\right) +
\left( \frac{\partial \alpha_j}{\partial x_n}\alpha_i - \frac{\partial \alpha_i}{\partial x_n} \alpha_j \right) =0\quad \forall i,j \in \{1,\ldots, n-1\} \right\}.
$$
For every tuple $I=(i_1,\ldots ,i_k) \in \{1,\ldots ,n-1\}^k$ we denote by
$f_I$ the $C^\infty$ vector field constructed by Lie brackets of
$f_1, f_2, \ldots ,f_{n-1}$ as follows,
$$
f_I = [f_{i_1},[f_{i_2},\ldots ,[f_{i_{k-1}},f_{i_k}]\ldots ]].
$$
We call $k=\mbox{length}(I)$ the length of the Lie bracket $f_I$. Since $\Delta$ is nonholonomic, there is some positive integer $r$ such that
$$
\R^n = \SPAN \left\{f_I (x) \mid \mbox{length} (I) \leq r \right\}\qquad \forall x\in \mathcal{V}.
$$
It is easy to see that, for every $I$ such that $\mbox{length}(I)
\geq 2$, there is a $C^\infty$ function $g_I :\mathcal{V}
\rightarrow \R$ such that
$$
f_I(x) =  g_I(x)  \frac{\partial}{\partial x_n} \qquad \forall x \in \mathcal{V}.
$$
Defining the sets $A_k$ as
$$
A_k :=\left\{ x\in \mathcal{V} \mid g_I(x)=0\quad \forall I \mbox{ such that length}(I) \leq k \right\},
$$
we have
$$
\Sigma_{\Delta} = \bigcup_{k=2}^r \left( A_k \setminus
A_{k+1}\right).
$$
We now observe that, thanks to the Implicit Function Theorem,
each set $A^k\setminus A^{k+1}$ can be covered by a countable
union of smooth hypersurfaces. Indeed assume that some given $x$
belongs to $A_k \setminus A_{k+1}$. This implies that there is
some $J=(j_1,\ldots ,j_{k+1})$ of length $k+1$ such that $g_J (x)
\neq 0$. Set $I=(j_2,\ldots ,j_{k+1})$. Since $g_I(x)=0$, we have
$$
g_J(x) = \left( \frac{\partial g_I }{\partial x_{j_1}}(x) + \frac{\partial g_I}{\partial x_n} (x) \alpha_{j_1} (x) \right) \frac{\partial}{\partial x_n}\neq 0.
$$
Hence, either $\frac{\partial g_I }{\partial x_{j_1}}(x) \neq 0$ or $ \frac{\partial{g_I}}{\partial x_n} (x) \neq 0$.\\
Consequently, we deduce that we have the following inclusion
$$
A^k\setminus A^{k+1} \subset \bigcup_{\mbox{length}(I)=k} \left\{
x\in \mathcal{V} \mid \exists \,i\in \{1,\ldots ,n\} \mbox{ such that
}\frac{\partial g_I }{\partial x_i}(x) \neq 0 \right\}.
$$
We conclude easily. Finally, the fact that any Goh path is
contained in $\Sigma_{\Delta}$ is obvious.

\end{document}